\documentclass[10pt]{article}%
\usepackage{amsmath}
\usepackage{amsfonts}
\usepackage{mathrsfs}
\usepackage{amssymb}
\usepackage[linkcolor=black,anchorcolor=black,citecolor=black]{hyperref}
\usepackage{graphicx}
\usepackage[body={15.5cm,21cm}, top=3cm]{geometry}%
\providecommand{\U}[1]{\protect\rule{.1in}{.1in}}
\providecommand{\U}[1]{\protect \rule{.1in}{.1in}}
\newtheorem{theorem}{Theorem}[section]

\newtheorem{corollary}[theorem]{Corollary}

\newtheorem{definition}[theorem]{Definition}

\newtheorem{lemma}[theorem]{Lemma}

\newtheorem{proposition}[theorem]{Proposition}
\newtheorem{remark}[theorem]{Remark}

\newenvironment{proof}[1][Proof]{\noindent \textbf{#1.} }{\  \rule{0.5em}{0.5em}}
\begin{document}

\title{On monotonicity and order-preservation for multidimensional $G$-diffusion processes}
\author{Peng Luo\thanks{School of Mathematics and Qilu Securities Institute for Financial Studies, Shandong University and Department of Mathematics and Statistics, University of Konstanz; pengluo1989@gmail.com. The author thanks the partial support from China Scholarship Council, and the ¡°111¡± project(No. B12023).}
\and Guangyan Jia\thanks{Qilu Securities Institute for Financial Studies, Shandong University; jiagy@sdu.edu.cn. The author thanks the partial support from the NSF of China (11171186), and the ¡°111¡± project(No. B12023).}}
\date{}
\maketitle

%%%%%%%%%%%%%%%%%%%%%%%%%%%%%%%%%%%%%%%%%%%%%%%%%%%%%%%%%%%%%%%%%%%%%%%%%%%%%%%%%%%%%%%%%%%%%%%%%%%%%%%%%%%%%%%

\begin{abstract}

In this paper, we prove a comparison theorem for multidimensional $G$-SDEs. Moreover we obtain respectively the sufficient conditions and necessary conditions of the monotonicity and order-preservation for two multidimensional $G$-diffusion processes. Finally, we give some applications.
\end{abstract}

\section{Introduction}
In the classical framework, it is known that It\^{o} diffusions can be used to construct linear semigroups, which are known as Markov semigroups. The relationships among It\^{o}'s diffusions, Markov semigroups and infinitesimal generators have been well studied and many interesting results have been deduced. They can be summarized as follows. We suppose $(X_t)_{t\geq 0}$ to be $n$-dimensional It\^{o} diffusion
\begin{equation*}
dX_t=b(X_t)dt+\sigma(X_t)dW_t,
\end{equation*}
where $(W_t)_{t\geq 0}$ is a $d$-dimensional Brownian motion and $b,~\sigma$ are Lipschitz continuous functions on $\mathbb{R}^n$. The Markov semigroup $P_t$ is defined by $P_tf(x)=E[f(X_{t}^{0,x})]$, where $X^{0,x}_{.}$ represents the It\^{o} process with initial condition $x$ at initial time $t=0$ and $f$ is a function defined on $\mathbb{R}^n$. Here $E[\cdot]$ stands for the expectation related to a probability $P$. The infinitesimal generator $L$ of the Markov semigroup, which satisfies
\begin{equation*}
Lf(x)=\lim\limits_{t\rightarrow 0^+}\frac{P_tf(x)-f(x)}{t}
\end{equation*}
for $f$ appropriately taken such that the above limit exists, is of the following form:
\begin{equation*}
Lf=\frac{1}{2}tr[\sigma^*\frac{\partial^2}{\partial x^2}f\sigma]+b^*\frac{\partial}{\partial x}f,
\end{equation*}
where "*" denotes the transposition. For more details, the readers can refer to, for example, Stroock and Varadhan (\cite{SV}) and Rogers and Williams (\cite{RW}).

The monotonicity property of the semigroups associated with the corresponding diffusions was initiated by Holly (\cite{HO}) and studied by Cox (\cite{CO}) and Harris (\cite{HA}). Afterwards, Herbst and Pitt (\cite{HP}) investigated the use of diffusion equations as a tool for establishing stochastic monotonicity of semigroups. Chen and Wang (\cite{CW}) continued the study on the order-preservation for multidimensional diffusion processes. One of the main results of Chen and Wang (\cite{CW}), which covers the monotonicity result in Herbst and Pitt (\cite{HP}), is as follows:
\begin{lemma}{(Chen and Wang (\cite{CW}, Theorem 1.3))}
Let $A=\frac{1}{2}\Sigma_{i,j=1}^{n}a_{ij}\frac{\partial^2}{\partial  x_i\partial x_j}+\Sigma_{i=1}^{n}l_i\frac{\partial}{\partial x_i}$ (resp. $\bar{A}=\frac{1}{2}\Sigma_{i,j=1}^{n}\bar{a}_{ij}\frac{\partial^2}{\partial  x_i\partial x_j}+\Sigma_{i=1}^{n}\bar{l}_i\frac{\partial}{\partial x_i}$). Assume that ($a_{ij}$) and ($\bar{a}_{ij}$) are nonnegative definite everywhere, $a_{ij},~\bar{a}_{ij},~l_i,~\bar{l}_i\in C(\mathbb{R}^n)$ and the martingale problems for $A$ and $\bar{A}$ are well posed. Let $P_t$ (resp. $\bar{P}_t$) be the Markov semigroup generated by $A$ (resp. $\bar{A}$). Then $P_t\geq\bar{P}_t$ if and only if the following two conditions hold:\\
(a) for all $i$ and $j$, $a_{ij}\equiv\bar{a}_{ij}$ and $a_{ij}(x)$ depends only on $x_i$ and $x_j$;\\
(b) for all $i$, $l_i(x)\geq\bar{l}_i(y)$ whenever $x_i=y_i$ and $x_j\geq y_j$ for all $j\in\{1,\ldots,n\}$ and $j\neq i$.
\end{lemma}

Peng (\cite{Peng1997}) introduced the notion of $g$-expectation defined via a backward stochastic differential equation (BSDE). A $g$-expectation preserves most properties of the classical expectations except nonlinearity since it is a nonlinear functional. Its nonlinearity can be characterized by its generator $g$. Zhang and Jia (\cite{ZG}) constructed a nonlinear semigroup by decoupled FBSDEs. They obtained the equivalent conditions of the monotonicity and order-preservation of semigroups.

Recently, Peng systemically established a time-consistent fully nonlinear
expectation theory ( see \cite{Peng2004}, \cite{Peng2005} and \cite{P09}). As a typical and important case, Peng  introduced the
$G$-expectation theory ( see \cite{P10}  and the references therein) in 2006.
In the  $G$-expectation framework ($G$-framework for short), the
notion of  $G$-Brownian motion and the corresponding stochastic
calculus of It\^{o}'s type were established. On that basis, Gao \cite{G} and Peng \cite {Peng4} have studied the existence and uniqueness of the solution of $G$-SDEs under a standard Lipschitz condition on its coefficients. Moreover, based on Gao \cite{G}, Bai and Lin \cite{BL} obtained the existence and uniqueness of the solution of GSDEs under some integral-Lipschitz conditions. For a recent account and development of this theory we refer the reader to (\cite{LX},\cite{LQ},\cite{LQ1},\cite{L},\cite{LW}). Recently, Hu et. al. (\cite{HJPS}) proved an existence and uniqueness result on BSDEs driven by $G-$Brownian motions (G-BSDEs), further (in \cite{HJPS1}) they gave a comparison theorem for G-BSDEs. He et. al. (\cite{HH}) proved the representation theorem for generators of G-BSDEs, and then the converse comparison theorem of G-BSDEs and some equivalent results for nonlinear expectations generated by G-BSDEs.

This paper is organized as follow: In section 2, we recall some notations and results that we will use in this paper. In section 3, we give our assumptions and recall some notations and results of $G$-SDES and $G$-BSDEs. In section 4, we obtain a comparison theorem for multidimensional $G$-SDEs. In section 5, we obtain respectively the sufficient conditions and necessary conditions of the monotonicity and order-preservation of two multidimensional $G$-diffusion processes. In section 6, some applications are given.
\section{Preliminaries}
\subsection{Sublinear expectation, $G$-Brownian motion and capacity}
\subsubsection{Sublinear expectation}
We present some preliminaries in the theory of sublinear
expectation, $G$-Brownian motions and the capacity under
$G$-framework. More details can be found in Peng \cite{Peng4} and Li and Peng \cite{LP}.

\begin{definition}
Let $\Omega$ be a given set and let $\mathcal{H}$ be a linear space
of real valued functions defined on $\Omega$ with $c\in\mathcal{H}$
for all constants $c$, and $|X|\in\mathcal{H}$, if $X\in\mathcal{H}$.
$\mathcal{H}$ is considered as the space of our random variables. A
sublinear expectation $\hat{\mathbb{E}}$ on $\mathcal{H}$ is a
functional $\hat{\mathbb{E}}: \mathcal{H}\rightarrow \mathbb{R}$
satisfying the following
properties: for all $X, Y\in\mathcal{H}$, we have\\
(a)Monotonicity: if $X\geq Y$, then $\hat{\mathbb{E}}[X]\geq\hat{\mathbb{E}}[Y].$\\
(b)Constant preserving: $\hat{\mathbb{E}}[c]=c$, $\forall \ c\in
\mathbb{R}.$\\
(c)Sub-additivity: $\hat{\mathbb{E}}[X+Y]\leq \hat{\mathbb{E}}[X]+\hat{\mathbb{E}}[Y].$\\
(d)Positive homogeneity: $\hat{\mathbb{E}}[\lambda
X]=\lambda\hat{\mathbb{E}}[X]$, $\forall\ \lambda\geq0.$

The triple $(\Omega,\mathcal{H},\hat{\mathbb{E}})$ is called a
sublinear expectation space.  $X \in\mathcal{ H}$ is called a random
variable in $(\Omega,\mathcal{H},\hat{\mathbb{E}})$. We often call $Y
= (Y_1, \ldots, Y_d), Y_i \in\mathcal{ H}$ a $d$-dimensional random
vector in $(\Omega,\mathcal{H},\hat{\mathbb{E}})$. Let us consider a
space of random variables $\mathcal{H}$ satisfying: if $X_i
\in\mathcal{H}, i = 1, \ldots, d,$ then $\varphi(X_1, \cdots,X_d)
\in\mathcal{H}$, for all $ \varphi\in
{C}_{b,Lip}(\mathbb{R}^d)$, where
${C}_{b,Lip}(\mathbb{R}^{d})$ is the space of all bounded real-valued
Lipschitz continuous functions. \end{definition}
\begin{definition}
In a sublinear expectation space
$(\Omega,\mathcal{H},\hat{\mathbb{E}})$ an $m$-dimensional random vector
$X=(X_1,\ldots,X_m)$ is said to be independent from another $n$-dimensional random vector $Y=(Y_1,\ldots,Y_n)$ under $\hat{\mathbb{E}}$ if for any test
function $\varphi\in{C}_{b,Lip}(\mathbb{R}^{m+n})$ we have
$$\hat{\mathbb{E}}[\varphi(X,Y)]=\hat{\mathbb{E}}[\hat{\mathbb{E}}[\varphi(x,Y)]_{x=X}]$$
\end{definition}
\begin{definition}
Let $X_{1}$ and $X_{2}$ be two $n$-dimensional random vector defined on sublinear
expectation spaces
$(\Omega_{1},\mathcal{H}_{1},\hat{\mathbb{E}}_{1})$ and
$(\Omega_{2},\mathcal{H}_{2},\hat{\mathbb{E}}_{2})$, respectively.
They are called identically distributed, denoted by $X_{1}\overset{d}{=} X_{2}$,
if
$$\hat{\mathbb{E}}_{1}[\varphi(X_{1})]=\hat{\mathbb{E}}_{2}[\varphi(X_{2})], \ \forall \ \varphi\in {C}_{b,Lip}(\mathbb{R}^n)$$
We call $\bar{X}$ an independent copy of $X$ if $\bar{X}\overset{d}{=}X$ and
$\bar{X}$ is independent from $X$.
\end{definition}
\begin{definition}[$G$-normal distribution]
An $m$-dimensional random vector $X=(X_1,\ldots,X_m)$ on a sublinear expectation space
$(\Omega,\mathcal{H},\hat{\mathbb{E}})$ is called (centralized)
$G$-normal distributed if for any $a, b\geq 0$
$$aX+b\bar{X}\overset{d}{=}\sqrt{a^{2}+b^{2}}X$$
where $\bar{X}$ is an independent copy of $X$. The letter $G$
denotes the function
$$G(A):=\frac{1}{2}\hat{\mathbb{E}}[(AX,X)]:\mathbb{S}(d)\mapsto\mathbb{R}. $$

\end{definition}

\subsubsection{$G$-Brownian motion}

\begin{definition}[$G$-Brownian motion] Let $G : \mathbb{S}(d) \mapsto\mathbb{R}$ be a given monotonic and sublinear function. A process $(B({t})\in\mathcal{H})_{t\geq
0}$ on a sublinear expectation space
$(\Omega,\mathcal{H},\hat{\mathbb{E}})$ is called a $G$-Brownian
motion if the following properties are
satisfied: \\
(a) $B(0)=0.$\\
(b) For each $t, s\geq 0$ the increment $B_{t+s}-B_{t}\overset{d}{=} \sqrt{s}X$  and independent from
$(B_{t_{1}},B_{t_{2}},...,B_{t_{n}})$ for each $n\in \mathbb{N}$,
$0\leq t_{1}\leq t_{2}\leq ...\leq t_{n}\leq t$, where $X$ is $G$-normal distributed.
\end{definition}

 We denote by
$\Omega=C_{0}^{d}(\mathbb{R}^{+})$ the space of all
$\mathbb{R}^{d}$-value continuous paths
$(\omega_{t})_{t\in\mathbb{R}^{+}}$, with $\omega_{0}=0$, equipped
with the distance
$$\rho(\omega^{1},\omega^{2}):=\sum_{i=1}^{\infty}2^{-i}[(\max\limits_{t\in[0,i]}|\omega^{1}(t)-\omega^{2}(t)|)\wedge 1].$$

We denote by $\mathcal{B}({\Omega})$ the Borel $\sigma$-algebra of $\Omega$.

We also denote, for each $t\in[0,\infty)$:

$\Omega_{t}:=\{\omega({\cdot\wedge t}):\omega\in\Omega\},$

$\mathcal{F}_{t}:=\mathcal {B}(\Omega_{t}),$

$L^{0}{(\Omega)}:$ the space of all $\mathcal
{B}(\Omega)$-measurable real function,

$L^{0}{(\Omega_{t})}:$ the space of all $\mathcal
{B}(\Omega_{t})$-measurable real function,

$B_{b}{(\Omega_t)}:$ all bounded elements in $L^{0}{(\Omega)}$,
$B_{b}{(\Omega_{t})}:=B_{b}{(\Omega)}\cap L^{0}{(\Omega_{t})},$

$C_{b}{(\Omega_t)}:$ all continuous elements in $B_{b}{(\Omega)}$,
$C_{b}{(\Omega_{t})}:=B_{b}{(\Omega)}\cap L^{0}{(\Omega_{t})}.$

In Peng \cite{Peng4}, a $G$-Brownian motion is constructed on a sublinear expectation
sapce $(\Omega,L_{G}^p,\hat{\mathbb{E}},(\hat{\mathbb{E}}_t)_{t\geq 0})$for $p = 1$, where   $L_{G}^p$
 is a Banach space
under the natural norm $\|X\|_p=\hat{\mathbb{E}}[|X|^p]^{1/p}$. In this space the corresponding
canonical process $B(t,\omega) = \omega(t), t \in[0,\infty)$, for $\omega\in\Omega$,
 is a $G$-Brownian motion. It is proved in Denis et al.\cite{DHP} that
$L^{0}{(\Omega)}\supset L_{G}^{p}(\Omega)\supset
C_{b}{(\Omega)}$, and there exists a weakly compact family $\mathcal
{P}$ of probability measures defined on $(\Omega, \mathcal
{B}(\Omega))$ such that
$$\hat{\mathbb{E}}[X]=\sup\limits_{P\in\mathcal{P}}E_{P}[X], \ \ X\in L_{G}^{1}{(\Omega)}.$$

We now introduce the natural choquet capacity
$$c(A):=\sup\limits_{P\in\mathcal{P}}P(A), \ \ A\in\mathcal
{B}(\Omega).$$
\begin{definition}A set $A\subset\Omega$ is polar if $c(A)=0$.  A
property holds $ `quasi$-$surely '$ (q.s.) if it holds outside a
polar set.
\end{definition}
\begin{definition}A real function $X$ on $\Omega$ is said to be quasi-continuous if for each $\varepsilon>0$,
there exists an open set $O$ with $c(O)<\varepsilon$ such that
$X|_{O^{c}}$ is continuous.
\end{definition}

Then $L_{G}^{p}(\Omega)$ can be characterized as follows:
$$L_{G}^{p}(\Omega)=\{X\in\mathbb{L}^{0}(\Omega)|\sup\limits_{P\in\mathcal{P}}E_{P}[|X|^{p}]<\infty, and\ X\ is\ c \ \text{quasi-surely continuous}\}.$$

We denote, for $p>0$,

$\mathcal{L}^{p}:=\{X\in L^{0}(\Omega):
\hat{\mathbb{E}}[|X|^{p}]=\sup\limits_{P\in\mathcal{P}}E_{P}[|X|^{p}]<\infty\};$

$\mathcal{N}^{p}:=\{X\in L^{0}(\Omega):
\hat{\mathbb{E}}[|X|^{p}]=0\};$

$\mathcal{N}:=\{X\in L^{0}(\Omega):\  $X=0$,\ c\text{-qusi surely}\
\text{(q.s.)}\}$

It is seen that $\mathcal{L}^{p}$ and $\mathcal{N}^{p}$ are linear
spaces and $\mathcal{N}^{p}=\mathcal{N}$, for each $p>0$. We define
the space ${L}^{p}(\Omega)=\mathcal{L}^{p}/\mathcal{N}$  as the
equivalence classes of $\mathcal{L}^{p}$ modulo equality in
$\|\cdot\|_p$. Similarly, we can define ${L}^{p}(\Omega_t)=
{L}^{p}(\Omega) \cap L^{0}{(\Omega_{t})}$. As usual, we do not make
the distinction between classes and their representatives.
\begin{definition}Let ${L}_{b}^{p}(\Omega)$ be the completion of
$B_{b}(\Omega)$ under the Banach norm
$\hat{(\mathbb{E}}[|X|^{p}])^{1/p}$.
\end{definition}

Then  we have the following characterisation (see
\cite{DHP}): for each $p\geq 1$,
$${L}_{b}^{p}(\Omega)=\{X\in {L}^{p}(\Omega): \lim\limits_{n\rightarrow\infty}\hat{\mathbb{E}}[|X|^{p}\mathbf{I}_{\{|X|>n\}}]=0\}.$$

\begin{definition} For $p\geq 1$ and $T\in\mathbb{R}^{+}$ be fixed.
Consider the following simple type of processes:
\begin{align*}
M_{G}^{0,p}(0,T)=&\{\eta:=\eta(t,\omega)=\sum_{j=0}^{N-1}\xi_{j}(\omega)I_{[t_{j},t_{j+1})}(t)\\
 &\forall\ N>0,\ 0=t_{0}<...<t_{N}=T,\  \xi_{j}(\omega)\in \mathbb{L}_{G}^p(\Omega_{t_{j}}),\
 j=0,1,2,...,N-1.\}.
\end{align*}
For each $p\geq 1$, we denote by $M_{G}^{p}(0,T)$ the completion of
$M_{G}^{0,p}(0,T)$ under the norm
$$||\eta||_{M^{p}_{G}(0,T)}=|\int_{0}^{T}\hat{\mathbb{E}}[|\eta(t)|^{p}]dt|^{1/p}.$$
\end{definition}
\begin{definition} For $p\geq 1$ and $T\in\mathbb{R}^{+}$ be fixed.
Consider the following simple type of processes:
\begin{align*}
M_{b}^{0}(0,T)=&\{\eta:=\eta_{t}(\omega)=\sum_{j=0}^{N-1}\xi_{j}(\omega)\mathbf{I}_{[t_{j},t_{j+1})}(t).\\
 &\forall\ N>0,\ 0=t_{0}<...<t_{N}=T,\  \xi_{j}(\omega)\in B_{b}(\Omega_{t_{j}}),\
 j=0,1,2,...,N-1.\}.
\end{align*}
For each $p\geq 1$, we denote by $M_{b}^{p}(0,T)$ the completion of
$M_{b}^{0}(0,T)$ under the norm $||\cdot||_{M^{p}_G(0,T)}$.
\end{definition}
\begin{definition}(Integration with respect to $G$-Brownian motion)
For each $\eta\in M_{G}^{0,2}(0,T)$ with the form
$$\eta_{t}(\omega)=\sum_{k=0}^{N-1}\xi_{k}(\omega)I_{[t_{k},t_{k+1})}(t)$$
define
$$I(\eta)=\int_{0}^{T}\eta(s)dB(s):=\sum_{k=0}^{N-1}\xi_{k}(B_{t_{k+1}^{N}}-B_{t_{k}^{N}}).$$
The mapping $I: M_{G}^{0,2}(0,T)\mapsto
\mathbb{L}^{2}_{G}(\Omega_{T})$ can be continuously extended to $I:
M_{G}^{2}(0,T)\mapsto \mathbb{L}^{2}_{G}(\Omega_{T})$. For each
$\eta\in M_{G}^{2}(0,T)$, the stochastic integral is defined by
$$I(\eta):=\int_{0}^{T}\eta(s)dB_s,\ \ \eta\in M_{G}^{2}(0,T).$$
\end{definition}
%\begin{definition}(Integration with respect to $\langle B\rangle$)
%Define a mapping $M_{G}^{0,1}(0,T)\mapsto\mathbb
%{L}^{1}_G(\Omega_{T})$ as follows:
%$$Q(\eta)=\int_{0}^{T}\eta(s)d\langle B\rangle(s):=\sum_{k=0}^{N-1}\xi_{k}[\langle B\rangle_{t_{k+1}^{N}}-\langle B\rangle_{t_{k}^{N}}].$$
%Then $Q$ can be uniquely extended to $M_{G}^{1}(0,T)\mapsto \mathbb{L}^{1}_{G}(\Omega_{T})$. We
%still denote this mapping by
%$$Q(\eta):=\int_{0}^{T}\eta(s)dB_s,\ \ \eta\in M_{G}^{1}(0,T).$$
%\end{definition}

We have the following general case.
\begin{definition}
For each $\eta\in M_{b}^{0}(0,T)$ with the form
$$\eta_{t}(\omega)=\sum_{k=0}^{N-1}\xi_{k}(\omega)\mathbf{I}_{[t_{k},t_{k+1})}(t),$$
define
$$I(\eta)=\int_{0}^{T}\eta(s)dB_s:=\sum_{k=0}^{N-1}\xi_{k}(B_{t_{k+1}^{N}}-B_{t_{k}^{N}}).$$
The mapping $\mathbf{I}: M_{b}^{0}(0,T)\mapsto {L}^{2}(\Omega_{T})$
can be continuously extended to $\mathbf{I}: M_{b}^{2}(0,T)\mapsto
{L}^{2}(\Omega_{T})$. For each $\eta\in M_{b}^{2}(0,T)$, the
stochastic integral is defined by
$$I(\eta):=\int_{0}^{T}\eta(s)dB_s,\ \ \eta\in M_{b}^{2}(0,T).$$
\end{definition}
%\begin{definition}
%Define a mapping $Q: M_{b}^{0}(0,T)\mapsto {L}^{1}(\Omega_{T})$ as
%follows:
%$$Q(\eta)=\int_{0}^{T}\eta(s)d\langle B\rangle(s):=\sum_{k=0}^{N-1}\xi_{k}[\langle B\rangle_{t_{k+1}^{N}}-\langle B\rangle_{t_{k}^{N}}].$$
%Then $Q$ can be uniquely extended to $M_{b}^{1}(0,T)$. We still
%denote this mapping by
%$$Q(\eta):=\int_{0}^{T}\eta(s)d\langle B\rangle_s,\ \ \eta\in M_{b}^{1}(0,T).$$
%\end{definition}

For notational simplicity, we denote by $B^i:=B^{e_i}$ the $i$th coordinate of the d-dimensional $G$-Brownian motion $B$, under a given orthonormal basis $(e_1,\ldots,e_d)$ of $\mathbb{R}^d$. We also denote by $B_{t}^{a}:=(a,B_t)$ for fixed $a\in\mathbb{R}^d$. Then $(B_{t}^a)_{t\geq 0}$ is a 1-dimensional $G$-Brownian motion with $\sigma_{aa^{T}}^{2}=\hat{\mathbb{E}}[(a,B_1)^2]$ and $\sigma_{-aa^{T}}^{2}=-\hat{\mathbb{E}}[-(a,B_1)^2]$. Let $a$ and $\bar{a}$ be two given vectors in $\mathbb{R}^d$. We can define
\begin{equation*}
\langle B^a\rangle_t:=(B^{a}_t)^2-\int_{0}^{t}B_{s}^adB_{s}^a
\end{equation*}
where $\langle B^a\rangle$ is called the quadratic variation process of $B^a$. We can also define mutual variation process by
\begin{equation*}
\langle B^a,B^{\bar{a}}\rangle_t:=\frac{1}{4}[\langle B^a+B^{\bar{a}}\rangle_t-\langle B^a-B^{\bar{a}}\rangle_t]
\end{equation*}
It\^o's integral with respect to $\langle B^a\rangle$ or $\langle B^i,B^{j}\rangle$ can be similarly defined.
By Li and Peng \cite{LP}, we have
\begin{lemma}
Let $X\in M_{b}^{p}(0,T)$. Then for each $\varepsilon >0$, there
exists $\delta >0$ such that for all $\eta\in M_{b}^{0}(0,T)$
satisfying $\hat{\mathbb{E}}\int_{0}^{T}|\eta_{t}|dt\leq\delta$ and
$|\eta_{t}(\omega)|\leq 1$, we have
$\hat{\mathbb{E}}\int_{0}^{T}|X_{t}|^{p}|\eta_{t}|dt\leq\varepsilon$.
\end{lemma}
\begin{definition}
A stopping time $\tau$ relative to the filtration $(\mathcal{F}_{t})$ is a map
on $\Omega$ with values in $[0,T]$, such that for every $t$,
\[
\{ \tau \leq t\} \in \mathcal{F}_{t}.
\]

\end{definition}
\begin{lemma}
For each stopping time $\tau\in[0,T]$, we have $I_{[0,\tau]}(\cdot)X \in M^p_b
(0, T )$, for each
$X \in M^p_b
(0, T )$.
\end{lemma}
\begin{definition} A process $(M_{t})_{t\geq 0}$ is called a
$G$-martingale if for each $t\in[0,T], M_{t}\in
\mathbb{L}_{G}^{1}(\Omega_{t})$ and for each $s\in[0,t)$, we have
$$\hat{\mathbb{E}}_s[M_{t}]=M_{s}.$$
\end{definition}
\begin{corollary}
For each $\eta\in M_G^2(0,T)$, the process $(\int_{0}^{T}\eta(s)d B_s)_{t\in[0,T]}$ is a $G$-martingale.
\end{corollary}

\section{$G$-SDEs and $G$-BSDEs}
%Let us first recall some notations.
%$C^{n}(\mathbb{R}),C^{n}_{b,lip}(\mathbb{R})$ will denote respectively the
%set of functions of class $C^{n}$ from $\mathbb{R}$ into $\mathbb{R}$,
% the set of those functions of class $C^{n}$ whose partial
%derivatives of order less than or equal to $n$ are bounded Lipschtiz continuous function (and hence the
%function itself grows at most linearly at infinity).
\subsection{$G$-SDEs}
We make use of the following assumptions on the generator $b,h$ and $\sigma$ of GSDE:
\begin{description}
\item[(H1)] $b,~h_{ij}$ and $\sigma_i$ are given $\mathbb{R}^n$-valued bounded Lipschitz continuous functions defined on $[0,T]\times\mathbb{R}^n$ which satisfy the Lipchitz condition, i.e., there exists some constant $K$ such that $|\varphi(t, x) - \varphi(t, y)| \leq K|x -y|$, for each $t \in [0, T ]$,
$x, y\in\mathbb{R}^n$, $\varphi = b, h_{ij}$ and $\sigma_i$ respectively, $i,j=1,\ldots,d$.
\item[(H2)] $b,~h_{ij}$ and $\sigma_i$ are given $\mathbb{R}^n$-valued bounded Lipschitz continuous functions defined on $\mathbb{R}^n$, i.e., there exists some constant $K$ such that $|\varphi(x) - \varphi(y)| \leq K|x -y|$, for each $x, y\in\mathbb{R}^n$, $\varphi = b, h_{ij}$ and $\sigma_i$ respectively, $i,j=1,\ldots,d$.
\end{description}
We consider the following SDE driven by a $d$-dimensional $G$-Brownian motion:
\begin{align}\label{GSDE}
X(t)=X(0)+\int^t_0b(s,X(s))ds+\int^t_0h_{ij}(s,X(s))d\langle B^i, B^j\rangle_s+\int^t_0\sigma_i(s,X(s))dB^{i}_s,\ t\in[0,T],
\end{align}
where the initial condition $X_0\in\mathbb{R}^n$  is a given constant.
\begin{theorem}
Under the assumption (H1), there exists a unique solution $X\in M^
2_G(0, T )$ of the stochastic differential equation (\ref{GSDE}).
\end{theorem}
\subsection{$G$-BSDEs}
In this subsection, we present some notations and results of $G$-BSDEs. More details can be found in \cite{HJPS} and \cite{HJPS1}.
\begin{definition}
\label{def2.6} For fixed $T>0$, let $M_{G}^{0}(0,T)$ be the collection of
processes in the following form: for a given partition $\{t_{0},\cdot
\cdot \cdot,t_{N}\}=\pi_{T}$ of $[0,T]$,
\[
\eta_{t}(\omega)=\sum_{j=0}^{N-1}\xi_{j}I_{[t_{j},t_{j+1})}(t),
\]
where $\xi_{j}\in L_{ip}(\Omega_{t_{j}})$, $j=0,1,2,\cdot \cdot \cdot,N-1$. For
$p\geq1$, we denote by $H_{G}^{p}(0,T)$, $M_{G}^{p}(0,T)$ the completion of
$M_{G}^{0}(0,T)$ under the norms $\Vert \eta \Vert_{H_{G}^{p}}=\{ \mathbb{\hat
{E}}[(\int_{0}^{T}|\eta_{s}|^{2}ds)^{p/2}]\}^{1/p}$, $\Vert \eta \Vert
_{M_{G}^{p}}=\{ \int_{0}^{T}\mathbb{\hat{E}}[|\eta_{s}|^{p}]ds\}^{1/p}$ respectively.
\end{definition}

For each $\eta \in M_{G}^{1}(0,T)$, we can define the integrals $\int_{0}%
^{T}\eta_{t}dt$ and $\int_{0}^{T}\eta_{t}d\langle B^{\mathbf{a}}%
,B^{\mathbf{\bar{a}}}\rangle_{t}$ for each $\mathbf{a}$, $\mathbf{\bar{a}}%
\in \mathbb{R}^{d}$. For each $\eta \in H_{G}^{p}(0,T;\mathbb{R}^{d})$ with
$p\geq1$, we can define It\^{o}'s integral $\int_{0}^{T}\eta_{t}dB_{t}$.

Let $S_{G}^{0}(0,T)=\{h(t,B_{t_{1}\wedge t},\cdot \cdot \cdot,B_{t_{n}\wedge
t}):t_{1},\ldots,t_{n}\in \lbrack0,T],h\in C_{b,Lip}(\mathbb{R}^{n+1})\}$. For
$p\geq1$ and $\eta \in S_{G}^{0}(0,T)$, set $\Vert \eta \Vert_{S_{G}^{p}}=\{
\mathbb{\hat{E}}[\sup_{t\in \lbrack0,T]}|\eta_{t}|^{p}]\}^{\frac{1}{p}}$.
Denote by $S_{G}^{p}(0,T)$ the completion of $S_{G}^{0}(0,T)$ under the norm
$\Vert \cdot \Vert_{S_{G}^{p}}$.

We only consider non-degenerate $G$-normal distribution, i.e.,
\begin{description}
\item[(H3)] There exists some $\underline{\sigma}^{2}>0$ such that $G(A)-G(B)\geq\underline{\sigma}^{2}\mathrm{tr}[A-B]$ for any $A\geq B$.
\end{description}
We consider the following type of $G$-BSDEs (in this paper we always use
Einstein convention):%
\begin{align}
%\begin{split}
Y_{t}  &  =\xi+\int_{t}^{T}f(s,Y_{s},Z_{s})ds+\int_{t}^{T}g_{ij}(s,Y_{s}%
,Z_{s})d\langle B^{i},B^{j}\rangle_{s}\nonumber \\
&  -\int_{t}^{T}Z_{s}dB_{s}-(K_{T}-K_{t}), \label{pr-eq1}%
%\end{split}
\end{align}%
where%

\[
f(t,\omega,y,z),g_{ij}(t,\omega,y,z):[0,T]\times \Omega_{T}\times
\mathbb{R}\times \mathbb{R}^{d}\rightarrow \mathbb{R}%
\]
satisfy the following properties:

\begin{description}
\item[(H4)] There exists some $\beta>1$ such that for any $y,z$,
$f(\cdot,\cdot,y,z),g_{ij}(\cdot,\cdot,y,z)\in M_{G}^{\beta}(0,T)$.

\item[(H5)] There exists some $L>0$ such that
\[
|f(t,y,z)-f(t,y^{\prime},z^{\prime})|+\sum_{i,j=1}^{d}|g_{ij}(t,y,z)-g_{ij}%
(t,y^{\prime},z^{\prime})|\leq L(|y-y^{\prime}|+|z-z^{\prime}|).
\]

\end{description}

For simplicity, we denote by $\mathfrak{S}_{G}^{\alpha}(0,T)$ the collection
of processes $(Y,Z,K)$ such that $Y\in S_{G}^{\alpha}(0,T)$, $Z\in
H_{G}^{\alpha}(0,T;\mathbb{R}^{d})$, $K$ is a decreasing $G$-martingale with
$K_{0}=0$ and $K_{T}\in L_{G}^{\alpha}(\Omega_{T})$.

\begin{definition}
\label{def3.1} Let $\xi \in L_{G}^{\beta}(\Omega_{T})$ and $f$ satisfy (H4) and
(H5) for some $\beta>1$. A triplet of processes $(Y,Z,K)$ is called a solution
of equation (\ref{pr-eq1}) if for some $1<\alpha \leq \beta$ the following
properties hold:

\begin{description}
\item[(a)] $(Y,Z,K)\in \mathfrak{S}_{G}^{\alpha}(0,T)$;

\item[(b)] $Y_{t}=\xi+\int_{t}^{T}f(s,Y_{s},Z_{s})ds+\int_{t}^{T}%
g_{ij}(s,Y_{s},Z_{s})d\langle B^{i},B^{j}\rangle_{s}-\int_{t}^{T}Z_{s}%
dB_{s}-(K_{T}-K_{t})$.
\end{description}
\end{definition}

\begin{theorem}
\label{the1.1} (\cite{HJPS}) Assume that $\xi \in L_{G}^{\beta}(\Omega_{T})$
and $f$, $g_{ij}$ satisfy (H4) and (H5) for some $\beta>1$. Then equation
(\ref{pr-eq1}) has a unique solution $(Y,Z,K)$. Moreover, for any $1<\alpha<\beta$
we have $Y\in S_{G}^{\alpha}(0,T)$, $Z\in H_{G}^{\alpha}(0,T;\mathbb{R}^{d})$
and $K_{T}\in L_{G}^{\alpha}(\Omega_{T})$.
\end{theorem}
We consider the following $G$-BSDEs:%
\begin{align*}
Y_{t}^{l,\xi}  &  =\xi+\int_{t}^{T}f^{l}(s,Y_{s}^{l,\xi},Z_{s}^{l,\xi}%
)ds+\int_{t}^{T}g_{ij}^{l}(s,Y_{s}^{l,\xi},Z_{s}^{l,\xi})d\langle B^{i}%
,B^{j}\rangle_{s}\\
&  -\int_{t}^{T}Z_{s}^{l,\xi}dB_{s}-(K_{T}^{l,\xi}-K_{t}^{l,\xi}),\ l=1,2,
\end{align*}
where $g_{ij}^{l}=g_{ji}^{l}$.

He and Hu \cite{HH} generalized the comparison theorem in \cite{HJPS1}.

\begin{proposition}
\label{con-pro1} Let $f^{l}$ and $g_{ij}^{l}$ satisfy (H4) and (H5) for some
$\beta>1$, $l=1,2$. If $f^{2}-f^{1}+2G((g_{ij}^{2}-g_{ij}^{1})_{i,j=1}%
^{d})\leq0$, then for each $\xi \in L_{G}^{\beta}(\Omega_{T})$, we have
$Y_{t}^{1,\xi}\geq Y_{t}^{2,\xi}$ for $t\in \lbrack0,T]$.
\end{proposition}
\subsection{Nonlinear Feynman-Kac Formula}

In this subsection, we give the nonlinear Feynman-Kac Formula which was studied
in Peng \cite{P10} for special type of $G$-BSDEs. We
consider the following type of $G$-FBSDEs:%
\begin{equation}
dX_{s}^{t,\xi}=b(s,X_{s}^{t,\xi})ds+h_{ij}(s,X_{s}^{t,\xi})d\langle
B^{i},B^{j}\rangle_{s}+\sigma_{j}(s,X_{s}^{t,\xi})dB_{s}^{j},\ X_{t}^{t,\xi
}=\xi, \label{App1}%
\end{equation}

\begin{align}
Y_{s}^{t,\xi}  &  =\Phi(X_{T}^{t,\xi})+\int_{s}^{T}f(r,X_{r}^{t,\xi}%
,Y_{r}^{t,\xi},Z_{r}^{t,\xi})dr+\int_{s}^{T}g_{ij}(r,X_{r}^{t,\xi}%
,Y_{r}^{t,\xi},Z_{r}^{t,\xi})d\langle B^{i},B^{j}\rangle_{r}\nonumber\\
&  -\int_{s}^{T}Z_{r}^{t,\xi}dB_{r}-(K_{T}^{t,\xi}-K_{s}^{t,\xi}),
\label{App2}%
\end{align}
where $b$, $h_{ij}$, $\sigma_{j}:[0,T]\times\mathbb{R}^{n}\rightarrow
\mathbb{R}^{n}$, $\Phi:\mathbb{R}^{n}\rightarrow\mathbb{R}$, $f$, $g_{ij}:$
$[0,T]\times\mathbb{R}^{n}\times\mathbb{R}\times\mathbb{R}^{d}\rightarrow
\mathbb{R}$ are deterministic functions and satisfy the following conditions:

\begin{description}
\item[(A1)] $h_{ij}=h_{ji}$ and $g_{ij}=g_{ji}$ for $1\leq i,j\leq d$;

\item[(A2)] $b$, $h_{ij}$, $\sigma_{j}$, $f$, $g_{ij}$ are continuous in $t$;

\item[(A3)] There exist a positive integer $m$ and a constant $L>0$ such that%
\[
|b(t,x)-b(t,x^{\prime})|+\sum_{i,j=1}^{d}|h_{ij}(t,x)-h_{ij}(t,x^{\prime
})|+\sum_{j=1}^{d}|\sigma_{j}(t,x)-\sigma_{j}(t,x^{\prime})|\leq
L|x-x^{\prime}|,
\]%
\[
|\Phi(x)-\Phi(x^{\prime})|\leq L(1+|x|^{m}+|x^{\prime}|^{m})|x-x^{\prime}|,
\]
\begin{align*}
&  |f(t,x,y,z)-f(t,x^{\prime},y^{\prime},z^{\prime})|+\sum_{i,j=1}^{d}%
|g_{ij}(t,x,y,z)-g_{ij}(t,x^{\prime},y^{\prime},z^{\prime})|\\
&  \leq L[(1+|x|^{m}+|x^{\prime}|^{m})|x-x^{\prime}|+|y-y^{\prime
}|+|z-z^{\prime}|].
\end{align*}

\end{description}

We define%
\[
u(t,x):=Y_{t}^{t,x},\ \ (t,x)\in\lbrack0,T]\times\mathbb{R}^{n}.
\]
\begin{remark}
\label{remA.3} It is important to note that $u(t,x)$ is a deterministic
function of $(t,x)$, because $b$, $h_{ij}$, $\sigma_{j}$, $\Phi$, $f$,
$g_{ij}$ are deterministic functions and $\tilde{B}_{s}:=B_{t+s}-B_{t}$ is a
$G$-Brownian motion.
\end{remark}

We now give the Feynman-Kac formula.

\begin{theorem}
\label{theA.9} Let $u(t,x):=Y_{t}^{t,x}$ for $(t,x)\in\lbrack0,T]\times
\mathbb{R}^{n}$. Then $u(t,x)$ is the unique viscosity solution of the
following PDE:%
\begin{equation}
\left\{
\begin{array}
[c]{l}%
\partial_{t}u+F(D_{x}^{2}u,D_{x}u,u,x,t)=0,\\
u(T,x)=\Phi(x),
\end{array}
\right.  \label{feynman}%
\end{equation}
where%
\begin{align*}
F(D_{x}^{2}u,D_{x}u,u,x,t)=  &  G(H(D_{x}^{2}u,D_{x}u,u,x,t))+\langle
b(t,x),D_{x}u\rangle\\
&  +f(t,x,u,\langle\sigma_{1}(t,x),D_{x}u\rangle,\ldots,\langle\sigma
_{d}(t,x),D_{x}u\rangle),
\end{align*}%
\begin{align*}
H_{ij}(D_{x}^{2}u,D_{x}u,u,x,t)=  &  \langle D_{x}^{2}u\sigma_{i}%
(t,x),\sigma_{j}(t,x)\rangle+2\langle D_{x}u,h_{ij}(t,x)\rangle\\
&  +2g_{ij}(t,x,u,\langle\sigma_{1}(t,x),D_{x}u\rangle,\ldots,\langle
\sigma_{d}(t,x),D_{x}u\rangle).
\end{align*}

\end{theorem}

\section{Comparison Theorem for Multidimensional GSDEs}
In the classical framework, comparison theorem for stochastic differential equations are well studied (see \cite{A}, \cite{GD}, \cite{Ge}, \cite{GS}, \cite{LL} and etc.). In particular, Gei{\ss} and Manthey (\cite{Ge}) obtained a comparison theorem for multidimensional SDEs. We also refer to Lin (\cite{LQ}) and Luo and Wang (\cite{LW}) for the comparison theorem of 1-dimensional $G$-SDEs. In this paper, we follow the idea of their proof to get our results. We consider the following SDEs driven by a $d$-dimensional $G$-Brownian motion:
\begin{equation*}\label{GSDE1}
X(t)=X(0)+\int^t_0b(s,X(s))ds+\int^t_0h_{ij}(s,X(s))d\langle B^i, B^j\rangle_s+\int^t_0\sigma_i(s,X(s))dB^{i}_s,\ t\in[0,T]
\end{equation*}
and
\begin{equation*}\label{GSDE2}
Y(t)=Y(0)+\int^t_0\bar{b}(s,Y(s))ds+\int^t_0\bar{h}_{ij}(s,Y(s))d\langle B^i, B^j\rangle_s+\int^t_0\sigma_i(s,Y(s))dB^{i}_s,\ t\in[0,T],
\end{equation*}
where the initial conditions $X(0),~Y(0)\in\mathbb{R}^n$  are given constants together with
\begin{equation*}
X_k(0)\leq Y_k(0),~~k=1,\ldots,n.
\end{equation*}

We now give a comparison theorem for multidimensional $G$-SDEs.
\begin{theorem}\label{com}
Suppose that the following two conditions hold.
\begin{description}
\item[(B1)] For any $t\in [0,T]$, and $i=1,\ldots, n,$ the inequality
\begin{equation*}
b_i(t,x)-\bar{b}_i(t,y)+G([(h_{lk})_i+(h_{kl})_i]_{l,k=1}^{d}(t,x)-[(\bar{h}_{lk})_i+(\bar{h}_{kl})_i]_{l,k=1}^{d}(t,y))\leq 0
\end{equation*}
are fulfilled, whenever $x_i=y_i$ and $x_j\leq y_j$ for all $j\neq i$.
\item[(B2)] $b,~h_{ij},~\sigma_i$ and $\bar{b},~\bar{h}_{ij},~\sigma_i$ satisfy (H1) and $(\sigma_i)_k$ depends only on $x_k$, for each $k=1,\ldots,n$, $i,j=1,\ldots,d,$ i.e.,
\begin{equation*}
|(\sigma_i)_k(t,x)-(\sigma_i)_k(t,y)|\leq K|x_k-y_k|
\end{equation*}
for all $t\in [0,T]$, $x,y\in\mathbb{R}^n$.\\
\end{description}
Then for all $t\in [0,T]$,
\begin{equation*}
X_k(t)\leq Y_k(t)~~k=1,\ldots,n~~~\textit{q.s.}
\end{equation*}
\end{theorem}
\begin{proof}
We first proof the theorem under the following condition (B1') instead of (B1).
\begin{description}
\item[(B1')] For any $t\in [0,T]$ and $i=1,\ldots, n$ the inequality
\begin{equation*}
b_i(t,x)-\bar{b}_i(t,y)+G([(h_{lk})_i+(h_{kl})_i]_{l,k=1}^{d}(t,x)-[(\bar{h}_{lk})_i+(\bar{h}_{kl})_i]_{l,k=1}^{d}(t,y))< 0
\end{equation*}
are fulfilled, whenever $x_i=y_i$ and $x_j\leq y_j$ for all $j\neq i$.
\end{description}
Define the stopping times
\begin{equation*}
\tau_j=\inf\{t>0: X_j(t)>Y_j(t)\}\wedge T,~~j=1,\ldots,n
\end{equation*}
and
\begin{equation*}
\tau:=\tau_1\wedge\cdots\wedge\tau_n.
\end{equation*}
Obviously, $X_j(\tau_j)=Y_j(\tau_j)$ and $X_j(\tau)\leq Y_j(\tau)$, $j=1,\ldots,n.$ Because of condition (B1'), the continuity of $b,~h_{ij},~\bar{b},~\bar{h}_{ij},$ $i,j=1,2,\ldots,d,$ and the continuity of $X$ and $Y$ there exists a stopping time $T\geq\kappa>\tau$ q.s. defined on $\{\tau<T\}$ such that
\begin{equation}\label{leq}
\begin{split}
&b_j(t,X(s))-\bar{b}_j(t,Y_1(s),\ldots,Y_{j-1}(s),X_j(s),Y_{j+1}(s),\ldots,Y_{d}(s))+G([(h_{lk})_j+(h_{kl})_j)]_{l,k=1}^{d}(t,X(s))\\
&-[(\bar{h}_{lk})_j+(\bar{h}_{kl})_j]_{l,k=1}^{d}(t,Y_1(s),\ldots,Y_{j-1}(s),X_j(s),Y_{j+1}(s),\ldots,Y_{d}(s)))< 0
\end{split}
\end{equation}
on $\{\tau_j=\tau<T\}$ for all $s\in [\tau,\kappa]$ q.s.. Actually, we can define
\begin{align*}
\kappa_1:= \inf\{t>\tau:&b_j(t,X(s))-\bar{b}_j(t,Y_1(s),\ldots,Y_{j-1}(s),X_j(s),Y_{j+1}(s),\ldots,Y_{d}(s))\\
&+G([(h_{lk})_j+(h_{kl})_j)]_{l,k=1}^{d}(t,X(s))\\
&-[(\bar{h}_{lk})_j+(\bar{h}_{kl})_j]_{l,k=1}^{d}(t,Y_1(s),\ldots,Y_{j-1}(s),X_j(s),Y_{j+1}(s),\ldots,Y_{d}(s)))> 0\}\wedge T,
\end{align*}
then we take $\kappa=\frac{\tau+\kappa_1}{2}$. It is easy to check that $\kappa$ satisfies the above condition.

 Now we define $\rho(x)=Kx$ for $x\geq 0$. Then $\int_{0^+}\rho^{-2}(u)du=\infty$. Hence there exists a strictly increasing sequence $\{a_n\}^{\infty}_{n=0}$ such that $a_0=1$, $\lim\limits_{n\rightarrow\infty}a_n=0$ (actually, $a_n=\frac{2}{n(n+1)K^2+2}$) and
\begin{equation*}
\int_{a_n}^{a_{n-1}}\rho^{-2}(u)du=n,~~~\text{for all}~n\geq 1.
\end{equation*}
Let $\psi_n$ be a continuous function such that its support is contained in $(a_n,a_{n-1})$, $0\leq\psi_n(u)\leq 2\rho^{-2}(u)n^{-1}$ and $\int_{a_n}^{a_{n-1}}\psi_n(u)du=1$.
Put
\begin{equation*}
\varphi_n(x)=\left\{
\begin{aligned}
&0,\ \ \ \ \ \ \ \ \ \ \ \ \ \ \ \ \ \ \ \ \ \ \ \ \ x\leq 0,\\
&\int_{0}^{x}\int_{0}^{y}\psi_n(u)dudy,\ \ x>0.
\end{aligned}
\right.
\end{equation*}
One can easily see that $\varphi_n$  is twice continuously differentiable, $\varphi_n(0)=0$ for $x\leq 0$, $0\leq\varphi^{\prime}_{n}(x)\leq 1$ and $\varphi_n(x)\uparrow x^+$ as $n\rightarrow\infty.$
%we have
%\begin{align*}
%X_k(t)-Y_k(t)=&X_k(0)-Y_k(0)+\int_{0}^{t}[b_k(s,X(s))-\bar{b}_k(s,Y(s))]ds\\
%&+\int_{0}^{t}[((h_{ij})_k(s,X(s))-\bar{h}_{ij})_k(s,Y(s))]d\langle B^i, B^j\rangle _s\\
%&+\int_{0}^{t}[(\sigma_i)_k(s,X(s))-(\sigma_i)_k(s,Y(s))]dB^{i}_s
%\end{align*}
Assume
\begin{equation}\label{ass}
c(\{\tau<T\})>0.
\end{equation}
An application of $G$-It\^{o}'s formula (see \cite{LP}) yields
\begin{align*}
&\varphi_n(X_k((\tau+t)\wedge\kappa)-Y_k((\tau+t)\wedge\kappa))\\
=&\varphi_n(X_k(\tau)-Y_k(\tau))\\
&+\int_{\tau}^{(\tau+t)\wedge\kappa}\varphi_{n}^{\prime}(X_k(s)-Y_k(s))[b_k(s,X(s))-\bar{b}_k(s,Y(s))]ds\\
&+\int_{\tau}^{(\tau+t)\wedge\kappa}\varphi_{n}^{\prime}(X_k(s)-Y_k(s))[(h_{ij})_k(s,X(s))-(\bar{h}_{ij})_k(s,Y(s))]d\langle B^i, B^j\rangle _s\\
&+\int_{\tau}^{(\tau+t)\wedge\kappa}\varphi_{n}^{\prime}(X_k(s)-Y_k(s))[(\sigma_i)_k(s,X(s))-(\sigma_i)_k(s,Y(s))]dB^{i}_s\\
&+\frac{1}{2}\int_{\tau}^{(\tau+t)\wedge\kappa}\varphi_{n}^{''}(X_k(s)-Y_k(s))[(\sigma_i)_k(s,X(s))-(\sigma_i)_k(s,Y(s))]^2d\langle B^i\rangle_s\\
=&\varphi_n(X_k(\tau)-Y_k(\tau\wedge T))\\
&+\int_{\tau}^{(\tau+t)\wedge\kappa}\varphi_{n}^{\prime}(X_k(s)-Y_k(s))[b_k(s,X(s))-\bar{b}_k(s,Y(s))\\
&+G([(h_{ij})_k+(h_{ji})_k]_{i,j=1}^{d}(s,X(s))-[(\bar{h}_{ij})_k+(\bar{h}_{ji})_k]_{i,j=1}^{d}(s,Y(s))]ds\\
&+V_{(\tau+t)\wedge\kappa}-V_{\tau}\\
&+\int_{\tau}^{(\tau+t)\wedge\kappa}\varphi_{n}^{\prime}(X_k(s)-Y_k(s))[(\sigma_i)_k(s,X(s))-(\sigma_i)_k(s,Y(s))]dB^{i}_s\\
&+\frac{1}{2}\int_{\tau\wedge T}^{(\tau+t)\wedge\kappa}\varphi_{n}^{''}(X_k(s)-Y_k(s))[(\sigma_i)_k(s,X(s))-(\sigma_i)_k(s,Y(s))]^2d\langle B^i\rangle_s
\end{align*}
where
\begin{align*}
V_t=&\int_{0}^{t}\varphi_{n}^{\prime}(X_k(s)-Y_k(s))[(h_{ij})_k(s,X(s))-(\bar{h}_{ij})_k(s,Y(s))]d\langle B^i, B^j\rangle _s\\
&-\int_{0}^{t}\varphi_{n}^{\prime}(X_k(s)-Y_k(s))G([(h_{ij})_k+(h_{ji})_k]_{i,j=1}^{d}(s,X(s))-[(\bar{h}_{ij})_k+(\bar{h}_{ji})_k]_{i,j=1}^{d}(s,Y(s))ds.
\end{align*}
It is easy to check that $(V_t)_{0\leq t\leq T}$ is a decreasing process. Thus
\begin{align*}
&\varphi_n(X_k((\tau+t)\wedge\kappa)-Y_k((\tau+t)\wedge\kappa))\\
\leq&\varphi_n(X_k(\tau)-Y_k(\tau))\\
&+\int_{\tau}^{(\tau+t)\wedge\kappa}\varphi_{n}^{\prime}(X_k(s)-Y_k(s))[b_k(s,X(s))-\bar{b}_k(s,Y(s))\\
&\ \ \ +G([(h_{ij})_k+(h_{ji})_k]_{i,j=1}^{d}(s,X(s))-[(\bar{h}_{ij})_k+(\bar{h}_{ji})_k]_{i,j=1}^{d}(s,Y(s))]ds\\
&+\int_{\tau}^{(\tau+t)\wedge\kappa}\varphi_{n}^{\prime}(X_k(s)-Y_k(s))[(\sigma_i)_k(s,X(s))-(\sigma_i)_k(s,Y(s))]dB^{i}_s\\
&+\frac{1}{2}\int_{\tau}^{(\tau+t)\wedge\kappa}\varphi_{n}^{''}(X_k(s)-Y_k(s))[(\sigma_i)_k(s,X(s))-(\sigma_i)_k(s,Y(s))]^2d\langle B^i\rangle_s\\
=&S_1(n)+S_2(n)+S_3(n)+S_4(n).
\end{align*}
Obviously, from the construction it follows that $S_1(n)=0$, $n=1,2,\ldots.,$
\begin{equation*}
\hat{\mathbb{E}}[S_3(n)\mathbf{1}_{\{\tau_k=\tau\}}]=0.
\end{equation*}
From (B2) we  derive
\begin{equation*}
\hat{\mathbb{E}}[|S_4(n)|]\leq\frac{Ct}{n}
\end{equation*}
Relation (\ref{leq}) implies
\begin{align*}
&S_2(n)\mathbf{1}_{\{\tau_k=\tau\}}\\
&\leq\mathbf{1}_{\{\tau_k=\tau\}}\int_{\tau}^{(\tau+t)\wedge\kappa}\varphi_{n}^{\prime}(X_k(s)-Y_k(s))[b_i(t,X(s))-\bar{b}_i(t,Y_1(s),\ldots,Y_{k-1}(s),X_k(s),Y_{k+1}(s),\ldots,Y_{d}(s))\\
&+G([(h_{ij})_k+(h_{ij})_k)]_{i,j=1}^{d}(t,X(s))\\
&-[(\bar{h}_{ij})_k+(\bar{h}_{ij})_k]_{i,j=1}^{d}(t,Y_1(s),\ldots,Y_{k-1}(s),X_k(s),Y_{k+1}(s),\ldots,Y_{d}(s)))]ds\\
&+\mathbf{1}_{\{\tau_k=\tau\}}\int_{\tau}^{(\tau+t)\wedge\kappa}\varphi_{n}^{\prime}(X_k(s)-Y_k(s))[\bar{b}_i(t,Y_1(s),\ldots,Y_{k-1}(s),X_k(s),Y_{k+1}(s),\ldots,Y_{d}(s))-\bar{b}_k(s,Y(s))\\
&+G([(\bar{h}_{ij})_k+(\bar{h}_{ji})_k]_{i,j=1}^{d}(s,Y_1(s),\ldots,Y_{k-1}(s),X_k(s),Y_{k+1}(s),\ldots,Y_{d}(s)))-[(\bar{h}_{ij})_k+(\bar{h}_{ji})_k]_{i,j=1}^{d}(s,Y(s))]ds\\
&\leq\mathbf{1}_{\{\tau_k=\tau\}}K\int_{\tau}^{(\tau+t)\wedge\kappa}[X_k(s)-Y_k(s)]^{+}ds\\
&\leq\mathbf{1}_{\{\tau_k=\tau\}}K\int_{0}^{t}[X_k((\tau+s)\wedge\kappa)-Y_k((\tau+s)\wedge\kappa)]^{+}ds.
\end{align*}
Consequently, as $n\rightarrow\infty$ we arrive at
\begin{equation*}
\hat{\mathbb{E}}[(X_k((\tau+t)\wedge\kappa)-Y_k((\tau+t)\wedge\kappa))^{+}\mathbf{1}_{\{\tau_k=\tau\}}]\\
\leq K\int_{0}^{t}\hat{\mathbb{E}}[(X_k((\tau+s)\wedge\kappa)-Y_k((\tau+s)\wedge\kappa))^{+}\mathbf{1}_{\{\tau_k=\tau\}}]ds.
\end{equation*}
In view of Gronwall's inequality this implies
\begin{equation*}
\hat{\mathbb{E}}[(X_k((\tau+t)\wedge\kappa)-Y_k((\tau+t)\wedge\kappa))^{+}\mathbf{1}_{\{\tau_k=\tau\}}]=0
\end{equation*}
and hence
\begin{equation*}
X_k((\tau+t)\wedge\kappa)\leq Y_k((\tau+t)\wedge\kappa) ~~~\textit{q.s.}
\end{equation*}
on$\{\tau_k=\tau\}$ for all $t\in[\tau,\kappa]$. This contradicts (\ref{ass}).

Now we consider the condition (B1). Let $\epsilon>0$ be arbitrarily chosen and define
\begin{equation*}
b_{k}^{\epsilon}:=b_k-\epsilon,~~k=1,\ldots,n.
\end{equation*}
From (B1) it follows immediately that $b^{\epsilon}$ satisfies condition (B1'). Consequently, we get for the corresponding solutions $X^{\epsilon}$
and $Y$ the relation
\begin{equation*}
X_{k}^{\epsilon}(t)\leq Y_k(t)~~\textit{q.s.}
\end{equation*}
for all $t\in [0,T]$, $k=1,\ldots,n.$ Choose a strictly decreasing sequence $(\epsilon_m)_{m\geq 1}$ with $\lim_{m\rightarrow\infty}\epsilon_m=0$. By the same arguments as above we get
\begin{equation*}
X_{k}^{\epsilon_1}(t)\leq X_{k}^{\epsilon_2}(t)\leq\ldots\leq Y_k(t) ~~\textit{q.s.}
\end{equation*}
as well as
\begin{equation*}
X_{k}^{\epsilon_1}(t)\leq X_{k}^{\epsilon_2}(t)\leq\ldots\leq X_k(t) ~~\textit{q.s.}
\end{equation*}
for all $t\in [0,T]$, $k=1,\ldots,n.$

Define
\begin{equation*}
\tilde{X}_k(t):=\lim_{m\rightarrow\infty}X_{k}^{\epsilon_m}(t)
\end{equation*}
for each $t\in [0,T]$. Obviously,
\begin{equation*}
\tilde{X}_k(t)\leq Y_k(t)~~\textit{q.s.}
\end{equation*}
for all $t\in [0,T]$, $k=1,\ldots,n.$
Moreover, we have,
\begin{align*}
\hat{\mathbb{E}}[\sup_{0\leq t\leq T}|X^{\epsilon_m}(t)-X(t)|^2]=&\hat{\mathbb{E}}[\sup_{0\leq t\leq T}|\int_{0}^{t}(b^{\epsilon_m}(s,X^{\epsilon_m}(s))-b(s,X(s)))ds\\
&+\int_{0}^{t}(h_{ij}(s,X^{\epsilon_m}(s))-h_{ij}(s,X(s))d\langle B^i,B^j\rangle)\\
&+\int_{0}^{t}(\sigma_i(s,X^{\epsilon_m}(s))-\sigma_i(s,X(s)))dB^{i}_s|^2]\\
&\leq 3(\hat{\mathbb{E}}[\sup_{0\leq t\leq T}|\int_{0}^{t}(b^{\epsilon_m}(s,X^{\epsilon_m}(s))-b(s,X(s)))ds|^2]\\
&+\hat{\mathbb{E}}[\sup_{0\leq t\leq T}|\int_{0}^{t}(h_{ij}(s,X^{\epsilon_m}(s))-h_{ij}(s,X(s))d\langle B^i,B^j\rangle)|^2]\\
&+\hat{\mathbb{E}}[\sup_{0\leq t\leq T}[|\int_{0}^{t}(\sigma_i(s,X^{\epsilon_m}(s))-\sigma_i(s,X(s)))dB^{i}_s|^2]).
\end{align*}
Applying BDG-inequality, we get
\begin{align*}
\hat{\mathbb{E}}[\sup_{0\leq t\leq T}|X^{\epsilon_m}(t)-X(t)|^2]&\leq C(\epsilon_{m}^2+\int_{0}^{T}\hat{\mathbb{E}}[|X^{\epsilon_m}(t)-X(t)|^2]dt)\\
&\leq C(\epsilon_{m}^2+\int_{0}^{T}\hat{\mathbb{E}}[\sup_{0\leq s\leq t}|X^{\epsilon_m}(s)-X(s)|^2]dt).
\end{align*}
By Gronwall's inequality, we have
\begin{equation*}
\hat{\mathbb{E}}[\sup_{0\leq t\leq T}|X^{\epsilon_m}(t)-X(t)|^2]\leq C|\epsilon_m|^2.
\end{equation*}
Up to a subsequence, still denoted as $(X^{\epsilon_m})_{m\geq 1}$, we have
\begin{equation*}
\lim_{m\rightarrow\infty}X^{\epsilon_m}(t)=X(t) ~~\textit{q.s.}
\end{equation*}
for all $t\in [0,T]$. This ends the proof.
\end{proof}
\begin{remark}
In general, the condition (B1) in the above theorem can not be replaced by the following condition:
\begin{description}
\item[(B1'')]
\begin{equation*}
\bar{b}_i(t,y)-b_i(t,x)+G([(\bar{h}_{lk})_i+(\bar{h}_{kl})_i]_{l,k=1}^{d}(t,y)-[(h_{lk})_i+(h_{kl})_i]_{l,k=1}^{d}(t,x))\geq 0
\end{equation*}
are fulfilled, whenever $x_i=y_i$ and $x_j\leq y_j$ for all $j\neq i$.
\end{description}
In fact, let $b_1=h_1=\bar{b}_1=\bar{h}_1=h_2=\bar{b}_2=0,$ $b_2=\frac{\overline{\sigma}^2+\underline{\sigma}^2}{2}$, $\bar{h}_2=1$ and $x=\bar{x}=0$, we consider 1-dimensional $G$-Brownian motion. One can show that $b,h,\bar{b},\bar{h}$ satisfy condition (B''). However we can not deduce that $\frac{\overline{\sigma}^2+\underline{\sigma}^2}{2}t\leq\langle B\rangle_t$, $q.s.$ whenever $\underline{\sigma}^2<\overline{\sigma}^2$.
\end{remark}
\section{Stochastic monotonicity and order-preservation}
Lin (\cite{LX}) defined the infinitesimal generator of $G$-SDEs and obtained the representation theorem under the Lipschitz condition. Similarly, we can obtain the relationships among $G$-It\^{o}'s diffusions, Markov semigroups and infinitesimal generators. They can be summarized as follows. We suppose $(X_t)_{t\geq 0}$ to be $n$-dimensional $G$-It\^{o} diffusion
\begin{align*}
dX_t=X_0+b(X_s)ds+h_{ij}(X_s)d\langle B^i,B^j\rangle_s+\sigma_j(X_s)dB^{j}_s,\ t\in[0,T],
\end{align*}
where $(B_t)_{t\geq 0}$ is a $d$-dimensional $G$-Brownian motion and $b,~h,~\sigma$ are Lipschitz continuous functions on $\mathbb{R}^n$. The Markov semigroup $\mathcal{E}_t$ is defined by $\mathcal{E}_tf(x)=\hat{\mathbb{E}}[f(X_{t}^{0,x})]$, where $X^{0,x}_{.}$ represents the $G$-It\^{o} process with initial condition $x$ at initial time $t=0$ and $f$ is a function defined on $\mathbb{R}^n$. The infinitesimal generator $L$ of the Markov semigroup, which satisfies
\begin{align*}
Lf(x)=\lim\limits_{t\rightarrow 0^+}\frac{\mathcal{E}_tf(x)-f(x)}{t}
\end{align*}
for $f$ appropriately taken such that the above limit exists, is of the following form:
\begin{align*}
Lf=\langle \partial_xf,b\rangle+G(\langle\partial_xf,h\rangle +\langle\partial_{xx}^2f\sigma,\sigma\rangle)
\end{align*}
where $\langle\partial_xf,h\rangle +\langle\partial_{xx}^2f\sigma,\sigma\rangle$ is a $d\times d$ symmetric matrix in $\mathbb{S}^d(\mathbb{R})$, defined by:
\begin{align*}
\langle\partial_xf,h\rangle +\langle\partial_{xx}^2f\sigma,\sigma\rangle:=[\langle\partial_xf,h_{ij}+h_{ji}\rangle +\langle\partial_{xx}^2f\sigma_i,\sigma_j\rangle]_{i,j=1}^d,
\end{align*}
\begin{equation*}
L=\sum_{i=1}^{n} b_i\frac{\partial}{\partial_{x_i}}+G([\sum_{i=1}^{n} (h_{lk}+h_{kl})_i\frac{\partial}{\partial_{x_i}}+\sum_{i,j=1}^{n}\sigma_{il}\sigma_{jk}\frac{{\partial}^2}{\partial_{x_i}\partial_{x_j}}]_{l,k=1}^{d}).
\end{equation*}

Now we introduce the following definitions, which are similar to that in Herbst and Pitt (\cite{HP}) and Chen and Wang (\cite{CW}). Let "$\leq$" denote the usual semi-order in $\mathbb{R}^n$.
\begin{description}
\item[(1)] A measurable function $f$ is called monotone if
\begin{equation*}
f(x)\leq f(\bar{x})~~\textit{for all}~~ x\leq\bar{x}.
\end{equation*}
Denote by $\mathcal{M}$ the set of all bounded Lipschitz continuous monotone functions.
\item[(2)] For two semigroups $\{\mathcal{E}_t\}_{0\leq t\leq T}$ and $\{\bar{\mathcal{E}}_t\}_{0\leq t\leq T}$, we write $\mathcal{E}_t\geq \bar{\mathcal{E}}_t$, if for all $f\in\mathcal{M}$, for all $x\geq\bar{x}$ and $0\leq t\leq T$,
\begin{equation*}
\mathcal{E}_tf(x)\geq \bar{\mathcal{E}}_tf(\bar{x}).
\end{equation*}
 If in addition, $\mathcal{E}_t=\bar{\mathcal{E}}_t$, we call $\mathcal{E}_t$ monotone.
\end{description}
Let
\begin{align*}
Lf=\langle \partial_xf,b\rangle+G(\langle\partial_xf,h\rangle +\langle\partial_{xx}^2f\bar{\sigma},\bar{\sigma}\rangle),
\end{align*}
\begin{align*}
\bar{L}f=\langle \partial_xf,\bar{b}\rangle+G(\langle\partial_xf,\bar{h}\rangle +\langle\partial_{xx}^2f\bar{\sigma},\bar{\sigma}\rangle),
\end{align*}
\begin{align*}
\bar{L}^{'}f=\langle \partial_xf,\bar{b}\rangle+G(\langle\partial_xf,\bar{h}\rangle +\langle\partial_{xx}^2f\sigma,\sigma\rangle)
\end{align*}
and let $\{\mathcal{E}_t\}_{0\leq t\leq T}$, $\{\bar{\mathcal{E}}_t\}_{0\leq t\leq T}$ and $\{\bar{\mathcal{E}}^{'}_t\}_{0\leq t\leq T}$be the semigroup generated by $L$, $\bar{L}$ and $\bar{L}^{'}$ respectively. And we always assume that $b,~h_{ij},~\sigma_i$ and $\bar{b},~\bar{h}_{ij},~\bar{\sigma}_i$ satisfy (H2) for each $i,j=1,\ldots,d$.

We have the following results.
\begin{theorem}\label{5.1}
Suppose the following conditions hold:
%$\mathcal{E}_t$ is monotone if and only if the following conditions hold:\\
\begin{description}
\item[(C1)] for all $i,j$, $\sigma_{li}\sigma_{kj}$ depends only on $x_i$ and $x_j$, $l,k=1,\ldots, d.$
\item[(C2)] for all $i$, $b_i(x)-b_i(y)+G([(h_{l,k})_i(x)+(h_{k,l})_i(x)]_{l,k=1}^{d}-[(h_{l,k})_i(y)+(h_{k,l})_i(y)]_{l,k=1}^{d}))\leq 0$ whenever $x\leq y$ with $x_i=y_i$.
\end{description}
then $\mathcal{E}_t$ is monotone¡£
\end{theorem}
\begin{proof}
Suppose (C1)and(C2) hold. By setting $\bar{b}=b$ and $\bar{h}=h$ in Theorem (\ref{com}), we have $\forall x\leq\bar{x}$, $X^{0,x}_t\leq X^{0,\bar{x}}_t$ q.s.. Then by the monotonicity of $f$, the results follows.
\end{proof}
\begin{theorem}\label{5.2}
If $\mathcal{E}_t$ is monotone, then the following conditions hold:
%$\mathcal{E}_t$ is monotone if and only if the following conditions hold:\\
\begin{description}
\item[(C1)] for all $i,j$, $\sigma_{li}\sigma_{kj}$ depends only on $x_i$ and $x_j$, $l,k=1,\ldots, d.$
\item[(C2')] for all $i$, $b_i(x)-b_i(y)+G([(h_{l,k})_i(x)+(h_{k,l})_i(x)]_{l,k=1}^{d}-[(h_{l,k})_i(y)+(h_{k,l})_i(y)]_{l,k=1}^{d}))\geq 0$ whenever $x\geq y$ with $x_i=y_i$.
\end{description}
\end{theorem}
\begin{proof}
Suppose that $\mathcal{E}_t$ is monotone,
\begin{align}\label{1}
\lim\limits_{t\rightarrow 0^+}\frac{1}{t}(\mathcal{E}_tf(x)-f(x))=\langle \partial_xf,b\rangle+G(\langle\partial_xf,h\rangle +\langle\partial_{xx}^2f\sigma,\sigma\rangle).
\end{align}
\begin{description}
\item[(a)] For given $i$, we take $x^{(1)}\leq x^{(2)}$ with $x_{i}^{(1)}=x_{i}^{(2)}$ and a sequence of functions $f_m\in\mathcal{M}\bigcap C_{b}^{\infty}$ ($m\in\mathbb{N}$) such that $f_m(x)=(x_i-x_{i}^{(1)}+1)^{2m+1}$ in a neighborhood of $\{x^{(1)},x^{(2)}\}$. We have $\lim_{t\rightarrow 0}\frac{1}{t}(\mathcal{E}_tf_m(x^{(1)})-f_m(x^{(1)}))\leq\lim_{t\rightarrow 0}\frac{1}{t}(\mathcal{E}_tf_m(x^{(2)})-f_m(x^{(2)}))$, i.e.,
\begin{align*}
&\langle \partial_xf(x^{(1)}),b(x^{(1)})\rangle+G(\langle\partial_xf(x^{(1)}),h(x^{(1)})\rangle +\langle\partial_{xx}^2f(x^{(1)})\sigma(x^{(1)}),\sigma(x^{(1)})\rangle)\\
&\leq\langle \partial_xf(x^{(2)}),b(x^{(2)})\rangle+G(\langle\partial_xf(x^{(2)}),h(x^{(2)})\rangle +\langle\partial_{xx}^2f(x^{(2)})\sigma(x^{(2)}),\sigma(x^{(2)})\rangle).
\end{align*}
Then we get
\begin{align*}
&G([\frac{(h_{l,k})_i(x^{(2)})+(h_{k,l})_i(x^{(2)})}{2m}+\sigma_{li}(x^{(2)})\sigma_{ki}(x^{(2)})]_{l,k=1}^{d}\\
&\ \ \
-[\frac{(h_{l,k})_i(x^{(1)})+(h_{k,l})_i(x^{(1)})}{2m}+\sigma_{li}(x^{(1)})\sigma_{ki}(x^{(1)})]_{l,k=1}^{d})\\
&\geq G([\frac{(h_{l,k})_i(x^{(2)})+(h_{k,l})_i(x^{(2)})}{2m}+\sigma_{li}(x^{(2)})\sigma_{ki}(x^{(2)})]_{j,k=1}^{d})\\
&\ \ \
-G([\frac{(h_{l,k})_i(x^{(1)})+(h_{k,l})_i(x^{(1)})}{2m}+\sigma_{li}(x^{(1)})\sigma_{ki}(x^{(1)})]_{l,k=1}^{d})\\
&\geq\frac{1}{2m}[b_i(x^{(1)})-b_i(x^{(2)})].
\end{align*}
So $[\sigma_{li}(v)\sigma_{ki}(v)]_{l,k=1}^{d}\geq[\sigma_{li}(u)\sigma_{ki}(u)]_{l,k=1}^{d}$ as $m\rightarrow\infty$. Replacing $f_m$ with $(x_i-x_{i}^{(1)}-1)^{2m+1}$ in a neighborhood of $\{x^{(1)},x^{(2)}\}$, we obtain the inverse inequality. Therefore, $[\sigma_{li}(v)\sigma_{ki}(v)]_{l,k=1}^{d}=[\sigma_{li}(u)\sigma_{ki}(u)]_{l,k=1}^{d}.$
\item[(b)] For given $i,j$ $i\neq j$, we take $x^{(1)}\leq x^{(2)}$ with $x_{i}^{(1)}=x_{i}^{(2)}$, $x_{j}^{(1)}=x_{j}^{(2)}$and a sequence of functions $f_m\in\mathcal{M}\bigcap C_{b}^{\infty}$ ($m\in\mathbb{N}$) such that $f_m(x)=(x_i+x_j-x_{i}^{(1)}-x_{j}^{(1)}+1)^{2m+1}$ in a neighborhood of $\{x^{(1)},x^{(2)}\}$. We have $\lim_{t\rightarrow 0}\frac{1}{t}(\mathcal{E}_tf_m(x^{(1)})-f_m(x^{(1)}))\leq\lim_{t\rightarrow 0}\frac{1}{t}(\mathcal{E}_tf_m(x^{(2)})-f_m(x^{(2)}))$, i.e.,
\begin{align*}
&\langle \partial_xf(x^{(1)}),b(x^{(1)})\rangle+G(\langle\partial_xf(x^{(1)}),h(x^{(1)})\rangle +\langle\partial_{xx}^2f(x^{(1)})\sigma(x^{(1)}),\sigma(x^{(1)})\rangle)\\
&\leq\langle \partial_xf(x^{(2)}),b(x^{(2)})\rangle+G(\langle\partial_xf(x^{(2)}),h(x^{(2)})\rangle +\langle\partial_{xx}^2f(x^{(2)})\sigma(x^{(2)}),\sigma(x^{(2)})\rangle).
\end{align*}
Then we get
\begin{align*}
&G([\frac{(h_{l,k})_i(x^{(2)})+(h_{k,l})_i(x^{(2)})+(h_{l,k})_j(x^{(2)})+(h_{k,l})_j(x^{(2)})}{2m}+\sigma_{li}(x^{(2)})\sigma_{kj}(x^{(2)})]_{l,k=1}^{d}\\
&\ \ \ -[\frac{(h_{l,k})_i(x^{(1)})+(h_{k,l})_i(x^{(1)})+(h_{l,k})_j(x^{(1)})+(h_{k,l})_j(x^{(1)})}{2m}+\sigma_{li}(x^{(1)})\sigma_{kj}(x^{(1)})]_{l,k=1}^{d})\\
&\geq G([\frac{(h_{l,k})_i(x^{(2)})+(h_{k,l})_i(x^{(2)})+(h_{l,k})_j(x^{(2)})+(h_{k,l})_j(x^{(2)})}{2m}+\sigma_{li}(x^{(2)})\sigma_{kj}(x^{(2)})]_{j,k=1}^{d})\\
&\ \ \ -G([\frac{(h_{l,k})_i(x^{(1)})+(h_{k,l})_i(x^{(1)})+(h_{l,k})_j(x^{(1)})+(h_{k,l})_j(x^{(1)})}{2m}+\sigma_{li}(x^{(1)})\sigma_{kj}(x^{(1)})]_{l,k=1}^{d})\\
&\geq\frac{1}{2m}[b_i(x^{(1)})+b_j(x^{(1)})-b_i(x^{(2)})-b_j(x^{(2)})].
\end{align*}
So $[\sigma_{li}(v)\sigma_{kj}(v)]_{l,k=1}^{d}\geq[\sigma_{li}(u)\sigma_{kj}(u)]_{l,k=1}^{d}$ as $m\rightarrow\infty$. Replacing $f_m$ with $(x_i+x_j-x_{i}^{(1)}-x_{j}^{(1)}+1)^{2m+1}$ in a neighborhood of $\{x^{(1)},x^{(2)}\}$, we obtain the inverse inequality. Therefore, $[\sigma_{li}(v)\sigma_{kj}(v)]_{l,k=1}^{d}=[\sigma_{li}(u)\sigma_{kj}(u)]_{l,k=1}^{d}$. Thus (C1) holds.
\item[(c)] For given $i$, we take $x^{(1)}\leq x^{(2)}$ with $x^{(1)}_i=x^{(2)}_i$ and $f_m\in\mathcal{M}\bigcap C_{b}^{\infty}$ ($m\in\mathbb{N}$) such that $f(x)=x_i$ in a neighborhood of $\{x^{(1)},x^{(2)}\}$.By (\ref{1}), we have
\begin{equation*}
b_i(x^{(1)})+G([(h_{l,k})_i(x^{(1)})+(h_{k,l})_i(x^{(1)})]_{l,k=1}^{d})\leq b_i(x^{(2)})+G([(h_{l,k})_i(x^{(2)})+(h_{k,l})_i(x^{(2)})]_{l,k=1}^{d}).
\end{equation*}
 Thus by the subadditivity, (C2') holds.
 \end{description}
\end{proof}
\begin{theorem}\label{oder}
If $\mathcal{E}_t\geq\bar{\mathcal{E}}_t$ then the following two conditions hold:
\begin{description}
\item[(D1)] for all $i,j$, $\sigma_{il}\sigma_{jk}\equiv\bar{\sigma}_{il}\bar{\sigma}_{jk}$ and $\sigma_{il}\sigma_{jk}$ depends only on $x_i$ and $x_j$, $l,k=1,\ldots, d.$
\item[(D2)] for all $i$, $b_i(x)-\bar{b}_i(y)+G([(h_{l,k})_i(x)+(h_{k,l})_i(x)]_{l,k=1}^{d}-[(\bar{h}_{l,k})_i(y)+(\bar{h}_{k,l})_i(y)]_{l,k=1}^{d})\geq 0$ whenever $x\geq y$ with $x_i=y_i$.
\end{description}
\end{theorem}

To prove the above theorem, we need some Lemmas.
\begin{lemma}\label{lem1}
If $\mathcal{E}_t\geq\bar{\mathcal{E}}_t$, then $Lf(x)\geq\bar{L}(y)$ for all $x\geq y$ and $f\in\mathcal{M}\bigcap C_{b}^{\infty}$ with $f(x)=f(y).$
\end{lemma}
\begin{proof}
Without loss of generality, assume that $f\geq 0$. Choose $m>0$ such that $\{z:|z|<m\}$ contains $x$ and $y$ and take $h\in C_{b}^{\infty}$ such that
\begin{equation*}
1\geq h(z)=\left\{
\begin{aligned}
&1,~~\text{if}~|z|\leq m,\\
&0,~~\text{if}~|x|\geq m+1,\\
&>0,~\text{otherwise}.
\end{aligned}
\right.
\end{equation*}
Set
\begin{equation*}
f_1=hf+a(1-h),~~f_2=hf,
\end{equation*}
where $a$ is a constant larger than the upper bound of $f$. Then $f_1,f_2\in C_{0}^{\infty}$, $f_1\geq f\geq f_2$ and $f_1=f=f_2$ on the set $\{z:|z|<m\}$. Since
\begin{equation*}
\mathcal{E}_tf(x)\geq\bar{\mathcal{E}}_tf(y),~~f(x)=f(y),
\end{equation*}
we have
\begin{equation*}
\frac{1}{t}[\mathcal{E}_tf_1(x)-f_1(x)]\geq\frac{1}{t}[\bar{\mathcal{E}}_tf_2(x)-f_2(x)].
\end{equation*}
The assertion now follows by letting $t\downarrow 0$.
\end{proof}
\begin{lemma}\label{lem2}
If $\mathcal{E}_t\geq\bar{\mathcal{E}}_t$, then (D2) holds.
\end{lemma}
\begin{proof}
For given $i$, let $u\leq v$ with $u_i=v_i$. Choose $f\in\mathcal{M}\bigcap C_{b}^{\infty}$ such that in a neighborhood of $\{u,v\}$,
\[
f(x)=x_i.
\]
Then by Lemma (\ref{lem1}), we get
\begin{equation*}
b_i(v)-\bar{b}_i(u)+G([(h_{l,k})_i(v)+(h_{k,l})_i(v)]_{l,k=1}^{d}-[(\bar{h}_{l,k})_i(u)+(\bar{h}_{k,l})_i(u)]_{l,k=1}^{d})\geq 0.
\end{equation*}
\end{proof}
\begin{lemma}\label{lem3}
If $\mathcal{E}_t\geq\bar{\mathcal{E}}_t$, then (D1) holds.
\end{lemma}
\begin{proof}
The proof consists of two steps.
\begin{description}
\item[(1)] For given $i$, let $u\leq v$ with $u_i=v_i$. Choose $f_m\in\mathcal{M}\bigcap C_{b}^{\infty}$ ($m\in\mathbb{N}$) such that in a neighborhood of $\{u,v\}$,
\begin{equation*}
f_m(x)=(x_i-u_i+1)^{2m+1}.
\end{equation*}
By Lemma (\ref{lem1}), we have
\begin{align*}
&G([\frac{(h_{l,k})_i(v)+(h_{k,l})_i(v)}{2m}+\sigma_{il}(v)\sigma_{ik}(v)]_{l,k=1}^{d}-[\frac{(\bar{h}_{l,k})_i(u)+(\bar{h}_{k,l})_i(u)}{2m}+\bar{\sigma}_{il}(u)\bar{\sigma}_{ik}(u)]_{l,k=1}^{d})\\
&\geq G([\frac{(h_{l,k})_i(v)+(h_{k,l})_i(v)}{2m}+\sigma_{il}(v)\sigma_{ik}(v)]_{j,k=1}^{d})-G([\frac{(\bar{h}_{l,k})_i(u)+(\bar{h}_{k,l})_i(u)}{2m}+\bar{\sigma}_{il}(u)\bar{\sigma}_{ik}(u)]_{l,k=1}^{d})\\
&\geq\frac{1}{2m}[\bar{b}_i(u)-b_i(v)].
\end{align*}
Since $m$ is arbitrary, we deduce that:
\begin{equation*}
[\sigma_{il}(v)\sigma_{ik}(v)]_{l,k=1}^{d}\geq[\bar{\sigma}_{il}(u)\bar{\sigma}_{ik}(u)]_{l,k=1}^{d}.
\end{equation*}
Replacing $f_m$ with $(x_i-u_i-1)^{2m+1}$ in the neighborhood of $\{u,v\}$, we obtain the inverse inequality. Therefore $\sigma_{il}(v)\sigma_{ik}(v)=\bar{\sigma}_{il}(u)\bar{\sigma}_{ik}(u)$.
\item[(2)] For given $i\neq j$ and $u\leq v$: $u_i=v_i,~u_j=v_j$, choose $f_m\in\mathcal{M}\bigcap C_{b}^{\infty}$ ($m\in\mathbb{N}$) such that in a neighborhood of $\{u,v\}$,
\begin{equation*}
f_m(x)=(x_i+x_j-u_i-u_j+1)^{2m+1}.
\end{equation*}
By (1) and Lemma (\ref{lem1}), we get
\begin{align*}
&G([\frac{(h_{l,k})_i(v)+(h_{k,l})_i(v)+h_{l,k})_j(v)+(h_{k,l})_j(v)}{2m}+\sigma_{il}(v)\sigma_{jk}(v)]_{l,k=1}^{d}\\
&\ \ \ -[\frac{(\bar{h}_{l,k})_i(u)+(\bar{h}_{k,l})_i(u)+\bar{h}_{l,k})_i(u)+(\bar{h}_{k,l})_j(u)}{2m}+\bar{\sigma}_{il}(u)\bar{\sigma}_{jk}(u)]_{l,k=1}^{d})\\
&\geq G([\frac{(h_{l,k})_i(v)+(h_{k,l})_i(v)+(h_{l,k})_i(v)+(h_{k,l})_j(v)}{2m}+\sigma_{il}(v)\sigma_{jk}(v)]_{j,k=1}^{d})\\
&\ \ \ -G([\frac{(\bar{h}_{l,k})_i(u)+(\bar{h}_{k,l})_i(u)+(\bar{h}_{l,k})_i(u)+(\bar{h}_{k,l})_j(u)}{2m}+\bar{\sigma}_{il}(u)\bar{\sigma}_{jk}(u)]_{l,k=1}^{d})\\
&\geq\frac{1}{2m}[\bar{b}_i(u)+\bar{b}_j(u)-b_i(v)-b_j(v)]
\end{align*}
and
\begin{equation*}
[\sigma_{il}(v)\sigma_{jk}(v)]_{l,k=1}^{d}\geq[\bar{\sigma}_{il}(u)\bar{\sigma}_{jk}(u)]_{l,k=1}^{d}.
\end{equation*}
Similarly, we have the inverse inequality and hence $\sigma_{il}(v)\sigma_{jk}(v)=\bar{\sigma}_{il}(u)\bar{\sigma}_{jk}(u).$
\end{description}
\end{proof}

\begin{proof}
[Proof of theorem \ref{oder}]The proof follows directly from Lemma(\ref{lem2}) and Lemma (\ref{lem3}).
\end{proof}
\begin{corollary}\label{oder1}
If $\mathcal{E}_t\geq\bar{\mathcal{E}}^{'}_t$ then the following two conditions hold:
\begin{description}
\item[(D3)] for all $i,j$, $\sigma_{il}\sigma_{jk}$ depends only on $x_i$ and $x_j$, $l,k=1,\ldots, d.$
\item[(D4)] for all $i$, $b_i(x)-\bar{b}_i(y)+G([(h_{l,k})_i(x)+(h_{k,l})_i(x)]_{l,k=1}^{d}-[(\bar{h}_{l,k})_i(y)+(\bar{h}_{k,l})_i(y)]_{l,k=1}^{d})\geq 0$ whenever $x\geq y$ with $x_i=y_i$.
\end{description}
\end{corollary}
\begin{corollary}\label{oder1}
Suppose the following two conditions hold:
\begin{description}
\item[(D3)] for all $i,j$, $\sigma_{il}\sigma_{jk}$ depends only on $x_i$ and $x_j$, $l,k=1,\ldots, d.$
\item[(D4')] for all $i$, $\bar{b}_i(x)-b_i(y)+G([(\bar{h}_{l,k})_i(x)+(\bar{h}_{k,l})_i(x)]_{l,k=1}^{d}-[(h_{l,k})_i(y)+(h_{k,l})_i(y)]_{l,k=1}^{d})\leq 0$ whenever $x\leq y$ with $x_i=y_i$.
\end{description}
then $\mathcal{E}_t\geq\bar{\mathcal{E}}^{'}_t$.
\end{corollary}
\begin{proof}
Suppose (D3)and (D4') hold. By Theorem (\ref{com}), we have $\forall x\geq\bar{x}$, $X^{0,x}_t\geq \bar{X}^{0,\bar{x}}_t$ q.s.. Then by the monotonicity of $f$, the results follows.
\end{proof}

%\begin{lemma}
%Suppose that () and () hold and are all bounded, If one of $\{\mathcal{E}_t\}$ and $\{\bar{\mathcal{E}}_t\}$ is monotone, then $\mathcal{E}_t\geq\bar{\mathcal{E}}_t$.
%\end{lemma}
%\begin{proof}
%(a) Without loss of generality, assume that $\{\bar{\mathcal{E}}_t\}$ is monotone. We can also assume that are smooth. Otherwise, choose $\varphi\in C^{\infty}_0(\mathbb{R})$, $\varphi\geq 0$, $\int_{\mathbb{R}}\varphi=1$ and define\\
%For the last inequality, we have used the monotonicity of $\bar{\mathcal{E}}_t$ and Lemma. Thus () holds for and , Clearly , every () satisfies (). Since and converge uniformly to respectively, by a convergence theorem, the proof of the lemma is reduced to the smooth case.\\
%(b) By () and (), we have
%\begin{equation}
%Lf(x)\geq\bar{L}f(x),~~~x\in\mathbb{R}^n
%\end{equation}
%for all $f\in\mathcal{M}\bigcap C^{\infty}_0$. On the other hand, since the coefficients of the operators are assumed by (a) to be bounded and smooth, we have the integration by parts formula.
%\end{proof}
%\begin{theorem}
%Let $u(t,x)=\hat{\mathbb{E}}[f(X_{t}^{0,x})]$, then $u$ is the unique viscosity solution of the following PDE:
%\begin{equation}
%\left\{
%\begin{aligned}
%&\partial_tu-F(\partial_xu,\partial^{2}_{xx}u)=0,\\
%&u(0,x)=f(x).
%\end{aligned}
%\right.
%\end{equation}
%where
%\begin{align*}
%F(\partial_xu,\partial^{2}_{xx}u)=\langle \partial_xu,b\rangle+G(\langle\partial_xu,h\rangle +\langle\partial_{xx}^2u\sigma,\sigma\rangle)
%\end{align*}
%\end{theorem}
\begin{corollary}\label{5.9}
If $\mathcal{E}_t\geq\bar{\mathcal{E}}_t$ then the following hold:
\begin{description}
\item[(D1)] for all $i,j$, $\sigma_{il}\sigma_{jk}\equiv\bar{\sigma}_{il}\bar{\sigma}_{jk}$ and $\sigma_{il}\sigma_{jk}$ depends only on $x_i$ and $x_j$, $l,k=1,\ldots, d$
\item[(D2')] for all $x,K\in\mathbb{R}^n$, $K\geq0$, $K^*(b(x)-\bar{b}(x))+G([K^*((h_{l,k})(x)+(h_{k,l})(x))]_{l,k=1}^{d}-[K^*((\bar{h}_{l,k})_i(x)+(\bar{h}_{k,l})_i(x))]_{l,k=1}^{d})\geq 0$.
\end{description}
\end{corollary}
\begin{proof}
First, we suppose that $\mathcal{E}_t\geq\bar{\mathcal{E}}_t$. Then (D1) holds directly from Theorem (\ref{oder}). For fixed $\bar{x}\in\mathbb{R}^n$, $K\in\mathbb{R}^n$, $K\geq 0$, we take $f\in\mathcal{M}\bigcap C_{b}^{\infty}$ such that in a neighborhood of $\bar{x}$, $f(x)=\sum_{i=1}^{n}K_i(x_i-\bar{x}_i)$. Then we have $\lim\limits{t\rightarrow 0}\frac{1}{t}(\mathcal{E}_tf(\bar{x})-f(\bar{x}))\geq\lim\limits{t\rightarrow 0}\frac{1}{t}(\bar{\mathcal{E}}_tf(\bar{x})-f(\bar{x}))$. Thus (D2')holds.
\end{proof}
\begin{theorem}\label{5.10}
Assume (H3) and assume that $\sigma\sigma^*$ (or resp. $\bar{\sigma}\bar{\sigma}^*$) is uniformly positive definite, i.e., there exists a constant $\beta>0$, such that for all $y\in\mathbb{R}^n$, $x\in\mathbb{R}^n$, $y^*\sigma(x)\sigma^*(x)y\geq\beta|y|^2$. If one of $\mathcal{E}_t$ and $\bar{\mathcal{E}}_t$ is monotone, if the following hold:
\begin{description}
\item[(D1)] for all $i,j$, $\sigma_{il}\sigma_{jk}\equiv\bar{\sigma}_{il}\bar{\sigma}_{jk}$ and $\sigma_{il}\sigma_{jk}$ depends only on $x_i$ and $x_j$, $l,k=1,\ldots, d$
\item[(D5)] for all $x,K\in\mathbb{R}^n$, $K\geq0$, $K^*(\bar{b}(x)-b(x))+G([K^*((\bar{h}_{l,k})(x)+(\bar{h}_{k,l})(x))]_{l,k=1}^{d}-[K^*((h_{l,k})_i(x)+(h_{k,l})_i(x))]_{l,k=1}^{d})\leq 0$.
\end{description}
 then $\mathcal{E}_t\geq\bar{\mathcal{E}}_t$.
\end{theorem}
\begin{proof}
We now suppose (D1) and (D5) hold. Without lost of generalization, we assume $\bar{\mathcal{E}}$ is monotone. Here we denote $(\sigma\sigma^{*})^{-1}\sigma$ by $\Sigma$. Let $f\in\mathcal{M}$ and we consider the following $G$-SDE and $G$-BSDEs:
\begin{equation*}
\bar{X}^{0,x}_{t}=x+\int_{0}^{t}\sigma_i(\bar{X}^{0,x}_{r})dB^{i}_r,\ \ \ t\in[0,T],
\end{equation*}
\begin{align*}
Y^{0,x,t,f}_{s}=f(\bar{X}^{0,x}_{t})&+\int_{s}^{t}b^{*}(\bar{X}^{0,x}_{r})\Sigma(\bar{X}^{0,x}_{r})Z^{0,x,t,f}_{r}dr+\int_{s}^{t}h_{ij}^{*}(\bar{X}^{0,x}_{r})\Sigma(\bar{X}^{0,x}_{r})Z^{0,x,t,f}_{r}d\langle B^i, B^j\rangle_r\\
&-\int_{s}^{t}Z^{0,x,t,f}_{r}dB_r-(K_s-K_t),\ \ \ s\in[0,t]
\end{align*}
and
\begin{align*}
\bar{Y}^{0,x,t,f}_{s}=f(\bar{X}^{0,x}_{t})&+\int_{s}^{t}\bar{b}^{*}(\bar{X}^{0,x}_{r})\Sigma(\bar{X}^{0,x}_{r})\bar{Z}^{0,x,t,f}_{r}dr+\int_{s}^{t}\bar{h}_{ij}^{*}(\bar{X}^{0,x}_{r})\Sigma(\bar{X}^{0,x}_{r})\bar{Z}^{0,x,t,f}_{r}d\langle B^i, B^j\rangle_r\\
&-\int_{s}^{t}\bar{Z}^{0,x,t,f}_{r}dB_r-(\bar{K}_s-\bar{K}_t),\ \ \ s\in[0,t].
\end{align*}
We have the following results: $u(t,x):=\mathcal{E}_tf(x)=Y^{0,x,t,f}_{0}$ is the unique viscosity solution of the following PDE:
\begin{equation}\label{PDE}
\left\{
\begin{aligned}
&\partial_tu-Lu=0,\\
&u(0,x)=f(x).
\end{aligned}
\right.
\end{equation}
%where
%\begin{align*}
%F(\partial_xu,\partial^{2}_{xx}u)=\langle \partial_xu,b\rangle+G(\langle\partial_xu,h\rangle +\langle\partial_{xx}^2u\sigma,\sigma\rangle).
%\end{align*}
$\bar{u}(t,x):=\bar{\mathcal{E}}_t=\bar{Y}^{0,x,t,f}_{0}$ is the unique viscosity solution of the following PDE:
\begin{equation}\label{PDE1}
\left\{
\begin{aligned}
&\partial_t\bar{u}-\bar{L}\bar{u}=0,\\
&\bar{u}(0,x)=f(x).
\end{aligned}
\right.
\end{equation}
%where
%\begin{align*}
%\bar{F}(\partial_xu,\partial^{2}_{xx}u)=\langle \partial_xu,\bar{b}\rangle+G(\langle\partial_xu,\bar{h}\rangle +\langle\partial_{xx}^2u\sigma,\sigma\rangle).
%\end{align*}
By Theorem 6.4.3 in Krylov \cite{Kr}(see also Theorem 4.4 in Appendix C in Peng \cite{Peng4}), there exists a constant $\alpha\in(0,1)$ such that for each $\kappa>0$,
\begin{equation*}
\|\bar{u}\|_{C^{1+\alpha/2,2+\alpha}([\kappa,T]\times\mathbb{R}^n)}<\infty.
\end{equation*}
Let $\hat{u}(t,x)=\bar{u}(T-t,x)$ and apply $G$-It\^{o}'s formula to $\hat{u}(s,\bar{X}_{s})$ for $s\in[0,T-\kappa]$
\begin{align*}
\hat{u}(s,\bar{X}^{0,x}_{s})=&\hat{u}(T-\kappa,\bar{X}^{0,x}_{T-\kappa})+\int_{s}^{T-\kappa}\bar{b}^{*}(\bar{X}^{0,x}_{r})\Sigma(\bar{X}^{0,x}_{r})\sigma^*(\bar{X}^{0,x}_{r})\partial_x\hat{u}(r,\bar{X}^{0,x}_r)dr\\
\ \ \ &+\int_{s}^{T-\kappa}\bar{h}_{ij}^{*}(\bar{X}^{0,x}_{r})\Sigma(\bar{X}^{0,x}_{r})\sigma^*(\bar{X}^{0,x}_{r})\partial_x\hat{u}(r,\bar{X}^{0,x}_r)d\langle B^i, B^j\rangle_r\\
\ \ \
&-\int_{s}^{T-\kappa}\sigma^*(\bar{X}^{0,x}_{r})\partial_x\hat{u}(r,\bar{X}^{0,x}_r)dB_r-(K^{'}_{T-\kappa}-K^{'}_s)
\end{align*}
where
 \begin{align*}
 K_s=&\frac{1}{2}\int_{0}^s\sigma^*(\bar{X}^{0,x}_{r})\partial_{xx}^2\hat{u}(r,\bar{X}^{0,x}_r)\sigma(\bar{X}^{0,x}_{r})d\langle B\rangle_r+\int_{s}^{T-\kappa}\bar{h}^{*}(\bar{X}^{0,x}_{r})\Sigma(\bar{X}^{0,x}_{r})\sigma^*(\bar{X}^{0,x}_{r})\partial_x\hat{u}(r,\bar{X}^{0,x}_r)d\langle B\rangle_r\\
 &-\int_{0}^{s}G(<\partial_x\hat{u}(r,\bar{X}^{0,x}_r),\bar{h}^{*}(\bar{X}^{0,x}_{r})>+<\partial_{xx}^2\hat{u}(r,\bar{X}^{0,x}_r)\sigma(\bar{X}^{0,x}_{r}),\sigma(\bar{X}^{0,x}_{r})>)dr
 \end{align*}
is a non-increasing $G$-martingale. By the uniqueness of solutions of $G$-BSDES, $\bar{Z}^{0,x,t,f}_{r}=\sigma^*(X^{0,x}_t)\partial_x\hat{u}(t,X^{0,x}_t)$ for all $0\leq r\leq t\leq T-\kappa$.
By the assumption that $\bar{\mathcal{E}}_t$ is monotone, we have $\bar{u}$ (then $\hat{u}$) is nondecreasing in $x$. Thus $\partial_x\hat(u)\geq 0$. By condition (D5) and the comparison of $G$-BSDEs, we have $Y^{0,x,t,f}_{r}\geq \bar{Y}^{0,x,t,f}_{r}$ for all $0\leq r\leq t\leq T-\kappa$. Particularly, we have $\mathcal{E}_tf(x)\geq\bar{\mathcal{E}}_tf(x)$. By the monotonicity of $\bar{\mathcal{E}}_t$, for all $x\geq\bar{x}$, $\mathcal{E}_tf(x)\geq\bar{\mathcal{E}}_tf(x)\geq\bar{\mathcal{E}}_tf(\bar{x})$. Thus $\mathcal{E}_t\geq\bar{\mathcal{E}}_t$ for all $t\in[0,T)$. By continuity, we have $\mathcal{E}_t\geq\bar{\mathcal{E}}_t$ for all $t\in[0,T]$.
\end{proof}
\section{Applications to PDEs}
In this section, we will give some applications of the above results to a special type of PDEs. We assume that $b,~h_{ij},~\sigma_i$ and $\bar{b},~\bar{h}_{ij},~\bar{\sigma}_i$ satisfy (H2) for each $i,j=1,\ldots,d$. We have the following results. We will omit the proofs since the results hold in view of Theorem \ref{5.1}, Theorem \ref{5.2}, Corollary \ref{5.9}, Theorem \ref{5.10} and the nonlinear Feynman-Kac formula in section 3.3.
\begin{theorem}
Let $u$ be the unique viscosity solution of PDE (\ref{PDE}), if (C1) and (C2) hold, then $u$ is nondecreasing in $x$  for all $f\in\mathcal{M}$.
\end{theorem}
\begin{theorem}
Let $u$ be the unique viscosity solution of PDE (\ref{PDE}), if $u$ is nondecreasing in $x$  for all $f\in\mathcal{M}$, then (C1) and (C2') hold.
\end{theorem}
\begin{theorem}
Let us consider PDE (\ref{PDE}) and PDE (\ref{PDE1}), if $u(t,x)\geq\bar{u}(t,\bar{x})$ for all $t\in[0,T]$, $f\in\mathcal{M}$ and $x\geq\bar{x}$, then (D1), (D2), and (D2') hold.
\end{theorem}
\begin{theorem}
Assume (H3) and assume that $\sigma\sigma^*$ (or resp. $\bar{\sigma}\bar{\sigma}^*$) is uniformly positive definite. Suppose that (C1) and (C2) hold for $b,h,\sigma$ or $\bar{b},\bar{h},
\bar{\sigma}$. Let us consider PDE (\ref{PDE}) and PDE (\ref{PDE1}), if (D1) and (D5) hold, then $u(t,x)\geq\bar{u}(t,\bar{x})$ for all $t\in[0,T]$, $f\in\mathcal{M}$ and $x\geq\bar{x}$.
\end{theorem}
\section*{Acknowledgment}
%The authors would like to thank the editor and a referee for their careful reading and helpful suggestions.
%%%%%%%%%%%%%%%%%%%%%%%%%%%%%%%%%%%%%%%%%%%%%%%%%%%%%%%%%%%%%%%%%%%%%%%%%%%%%%%%%%%

\end{document}